\documentclass{amsart}
\usepackage{latexsym,amsxtra,amscd,ifthen}
\usepackage{amsfonts}
\usepackage{verbatim}
\usepackage{amsmath}
\usepackage{amsthm}
\usepackage{amssymb}
\usepackage{enumerate}
\usepackage{url}
\usepackage{soul}

\theoremstyle{plain}

\newtheorem{theorem}{Theorem}
\newtheorem{lemma}[theorem]{Lemma}
\newtheorem{proposition}[theorem]{Proposition}
\newtheorem{corollary}[theorem]{Corollary}
\newtheorem{conjecture}[theorem]{Conjecture}

\numberwithin{theorem}{section}
\numberwithin{equation}{theorem}

\theoremstyle{definition}
\newtheorem{definition}[theorem]{Definition}
\newtheorem{notation}[theorem]{Notation}
\newtheorem{example}[theorem]{Example}
\newtheorem{remark}[theorem]{Remark}
\newtheorem{question}[theorem]{Question}
\newtheorem*{question*}{Question}
\newtheorem{hypothesis}[theorem]{Hypothesis}

\newcommand{\N}{\mathbb{N}}
\newcommand{\R}{\mathbb{R}}

\DeclareMathOperator{\op}{op}

\DeclareMathOperator{\End}{End}
\DeclareMathOperator{\Ext}{Ext}
\DeclareMathOperator{\Tor}{Tor}
\DeclareMathOperator{\Spec}{Spec}
\DeclareMathOperator{\Hom}{Hom}

\DeclareMathOperator{\Aut}{Aut}

\DeclareMathOperator{\ann}{Ann}
\DeclareMathOperator{\htp}{ht}
\DeclareMathOperator{\im}{Im}

\DeclareMathOperator{\GK}{GKdim}

\DeclareMathOperator{\K}{Kdim}

\DeclareMathOperator{\Mod}{Mod}
\DeclareMathOperator{\QMod}{QMod}
\DeclareMathOperator{\qmod}{qmod}
\DeclareMathOperator{\modu}{mod}

\DeclareMathOperator{\dep}{dep}
\DeclareMathOperator{\CM}{CM}
\DeclareMathOperator{\pd}{projdim}
\DeclareMathOperator{\injdim}{injdim}
\DeclareMathOperator{\gldim}{gldim}
\DeclareMathOperator{\add}{add}
\DeclareMathOperator{\proj}{proj}

\DeclareMathOperator{\refl}{ref}
\DeclareMathOperator{\Hm}{H}

\DeclareMathOperator{\coker}{coker}

\begin{document}

\title{Noncommutative quasi-resolutions}

\author{X.-S. Qin, Y.-H. Wang and J.J. Zhang}

\address{Qin: Shanghai Center for Mathematical Sciences,
School of Mathematical Sciences, Fudan
University, Shanghai 200433, China}

\email{13110840002@fudan.edu.cn}

\address{Wang: School of Mathematics, Shanghai key Laboratory
of Financial Information Technology, Shanghai University of
Finance and Economics, Shanghai 200433, China}

\email{yhw@mail.shufe.edu.cn}

\address{Zhang: Department of Mathematics, Box 354350,
University of Washington, Seattle, Washington 98195, USA}

\email{zhang@math.washington.edu}

\begin{abstract}
The notion of a noncommutative quasi-resolution is
introduced for a noncommutative noetherian algebra
with singularities, even for a non-Cohen-Macaulay algebra.
If $A$ is a commutative normal Gorenstein domain, then a
noncommutative quasi-resolution of $A$ naturally produces a
noncommutative crepant resolution (NCCR) of $A$ in the sense
of Van den Bergh, and vice versa. Under some mild hypotheses,
we prove that
\begin{enumerate}
\item[$(i)$]
in dimension two, all noncommutative quasi-resolutions of a
given noncommutative algebra are Morita equivalent, and
\item[$(ii)$]
in dimension three, all noncommutative quasi-resolutions of
a given noncommutative algebra are derived equivalent.
\end{enumerate}
These assertions generalize important results of Van den Bergh,
Iyama-Reiten and Iyama-Wemyss in the commutative and
central-finite cases.
\end{abstract}




\maketitle


\section*{Introduction}

A famous conjecture of Bondal-Orlov \cite{BO1, BO2} in birational
geometry states

\begin{conjecture} \cite{BO1, BO2}
\label{xxcon0.1}
If $Y_1$ and $Y_2$ are crepant resolutions of a
scheme $X$, then derived categories $D^b({\rm{coh}}(Y_1))$
and $D^b({\rm{coh}}(Y_2))$ are equivalent.
\end{conjecture}

In dimension three (respectively, two), this conjecture was proved by
Bridgeland \cite{Br} in 2002 (respectively, by Kapranov-Vasserot
\cite{KV} in 2000). The conjecture is still open in higher dimensions.
Noticed by Van den Bergh \cite{VdB1} in the study of one-dimensional
fibres and by Bridgeland-King-Reid \cite{BKR} in the study of the McKay
correspondence for dimension $d\leq3$ that both $D^b({\rm{coh}}(Y_1))$ and
$D^b({\rm{coh}}(Y_2))$ are equivalent to the derived category of certain
noncommutative rings. Motivated by Conjecture \ref{xxcon0.1} and work of
\cite{BKR, Br, VdB1}, Van den Bergh \cite{VdB2} introduced the notation
of a noncommutative crepant resolution (NCCR) of a commutative normal
Gorenstein domain $A$ (in the original reference, the author used the 
notation $R$). Let us recall the definition of a NCCR given in
\cite[Section 8]{IR} which is quite close to the original definition
of Van den Bergh \cite[Definition 4.1]{VdB2}. As usual, CM stands for
Cohen-Macaulay.

\begin{definition}
\label{xxdef0.2}
Let $R$ be a noetherian commutative CM ring and let $A$ be a
module-finite $R$-algebra.
\begin{enumerate}
\item[$(1)$] \cite{Au}
$A$ is called an {\it $R$-order} if $A$ is a maximal CM
$R$-module. An $R$-order $A$ is called {\it non-singular} if
$\gldim A_{\mathfrak{p}}=\K R_{\mathfrak{p}}$
for all ${\mathfrak{p}}\in\Spec(R).$
\item[$(2)$] \cite[Section 8]{IR}
Let $M$ be a finitely generated right $A$-module that is reflexive.
We say that $M$ gives a {\it noncommutative crepant resolution, or NCCR},
$B:=\End_{A}(M)$ of $A$ if
\begin{enumerate}
\item[$(i)$]
$M$ is a height one progenerator of $A$ 
(namely, $M_{\mathfrak{p}}$ is a progenerator of $A_{\mathfrak{p}}$ 
for any height one prime ideal $\mathfrak{p}$ of $R$), and
\item[$(ii)$]
$B$ is a non-singular $R$-order.
\end{enumerate}
\end{enumerate}
\end{definition}

Note that Van den Bergh's original definition of a NCCR was only for
$A=R$ being Gorenstein, since these are the types of varieties
which have a chance of admitting crepant resolutions and so there is
a good analogy with geometry \cite[after Definition 1.2]{IW1}.
However when $A$ is non-Gorenstein (but CM) there are sometimes many
NCCRs of $A,$ and these are related to cluster tilting (CT) objects
in the category of CM modules over $A$ \cite[Corollary 5.9]{IW2}.
Thus, although geometrically we are only really interested in NCCRs
when $A$ is Gorenstein, there are strong algebraic reasons to consider
the more general case. In this paper we will further relax the
hypotheses on $A$: we allow $A$ to be non-CM and to be noncommutative
in the most general sense. Van den Bergh made the following conjecture,
which is an extension of Bondal-Orlov Conjecture \ref{xxcon0.1}.

\begin{conjecture} \cite[Conjecture 4.6]{VdB2}
\label{xxcon0.3}
If $A$ is a normal Gorenstein domain, then all crepant resolutions of
$\Spec(A)$ {\rm{(}}commutative and noncommutative{\rm{)}} are derived
equivalent.
\end{conjecture}

Van den Bergh proved this conjecture for 3-dimensional terminal
singularities \cite[Theorem 6.6.3]{VdB2}. Since the existence of
commutative crepant resolutions is not equivalent to the existence
of noncommutative crepant resolutions in high dimension \cite{IW1},
one should probably break up the above conjecture into two parts:
commutative crepant resolutions and noncommutative crepant resolutions.
In \cite[Section 8]{IR}, Iyama-Reiten proved the noncommutative part
of this conjecture for noncommutative algebras $A$ as in
Definition \ref{xxdef0.2} in dimension three. Similarly in
\cite{IW1}, Iyama-Wemyss proved Conjecture \ref{xxcon0.3} for
NCCRs for CM algebras $A(=R)$ in dimension three, therefore generalizing
\cite[Corollary 8.8]{IR} to algebras which do not have
Gorenstein base rings.

\begin{theorem}
\label{xxthm0.4}
\begin{enumerate}
\item[$(1)$]
\cite[Corollary 8.8]{IR}
Let $R$ be a commutative normal Gorenstein domain with $\K R\leq3$ and
$A$ a module-finite $R$-algebra. Then all NCCRs of $A$
are derived equivalent.
\item[$(2)$]
\cite[Theorem 1.5]{IW1}
Let $A$ be a  $d$-dimensional $\CM$ equi-codimensional commutative normal domain
with a canonical module.
\begin{enumerate}
\item[$(2a)$]
If $d=2,$ then all NCCRs of $A$ are Morita equivalent.
\item[$(2b)$]
If $d=3,$ then all NCCRs of $A$ are derived equivalent.
\end{enumerate}
\end{enumerate}
\end{theorem}

Iyama-Wemyss \cite[Theorem 1.7]{IW1} also gave a sufficient condition
in arbitrary dimension ($d=\K A$) to establish when any two given
NCCRs of $R$ are derived equivalent.

The study of noncommutative singularities naturally leads a question of
how to deal with algebras that are not module-finite over their centers.
Such questions were implicitly asked in \cite{CKWZ1, CKWZ2}. Before we
present our solution, we would like to discuss some major differences
between the commutative and the noncommutative settings.

\begin{enumerate}
\item[($\bullet$)]
First of all, some of previous approaches of using moduli \cite{Br}
need to be re-developed in order to prove equivalences of derived
categories in a general noncommutative setting. However, very
little is known about the theory of general noncommutative moduli.
\item[($\bullet$)]
We say an algebra is {\it central-finite} if it is module-finite
over its center (or more precisely, a finite module over its
center). When algebras are not central-finite, some homological
tools fail due to the fact that localization does not work well
in the noncommutative setting. Our idea is to work with global
structures without going to the localization. For example, we
use Auslander regular algebras instead of algebras having
finite global dimension. In the commutative case, the Auslander
condition is automatic. It is easy to see that Auslander regular
algebras are a natural generalization of homologically
homogeneous algebras which are used in Van den Bergh's definition
of a NCCR.
\item[($\bullet$)]
In the commutative case, the Krull dimension, denoted by $\K$
(see \cite[Ch.6]{MR}), is used extensively and is implicitly
assumed. It is well-known that the Krull dimension might not
be a good dimension function in the noncommutative case.
Sometimes Gelfand-Kirillov dimension, denoted by $\GK$ (see
\cite{KL} and \cite[Ch.8]{MR}), is better than the Krull
dimension, and at other times vice versa. In the noncommutative
case, it is necessary to consider an abstract dimension function
(or several different ones in the different settings).
\end{enumerate}

Let $\partial$ be an exact symmetric dimension function in the sense
of Definition \ref{xxdef1.2} and Hypothesis \ref{xxhyp1.3}(3) (or
\cite[Section 6.8.4]{MR}) which is defined for all right $A$-modules
where $A$ is an algebra. Let $D$ be another algebra. Two right
$A$-modules (respectively, $(D,A)$-bimodules) $M$ and $N$ are called
{\it $s$-isomorphic} if there are a third right $A$-module
(respectively, $(D,A)$-bimodule) $P$ and two right $A$-module
(respectively, $(D,A)$-bimodule) maps
$$f: M\to P, \quad {\text{and}}\quad g: N\to P$$
such that the kernel and the cokernel of $f$ and $g$ (viewed as
right $A$-modules) have $\partial$-dimension less than or equal
to $s$. In this case, we write $M\cong_s N$. We refer to Definition
\ref{xxdef1.5} for more details. To state our main result without
going to too much details, we give a definition of a noncommutative
quasi-resolution in the following special case. Some technical
details are explained in Section \ref{xxsec3}.

\begin{definition}
\label{xxdef0.5}
We fix the dimension function $\partial$ to be $\GK$. Let $A$ be
a noetherian locally finite ${\mathbb N}$-graded algebra with
$\GK(A)=d\in {\mathbb N}$. If there are a noetherian locally
finite ${\mathbb N}$-graded Auslander regular CM algebra $B$
(see Definitions \ref{xxdef2.1} and \ref{xxdef2.3}) with
$\GK(B)=d$ and two ${\mathbb Z}$-graded bimodules $_{B}M_{A}$ and
$_{A}N_{B}$, finitely generated on both sides, such that
$$M\otimes_{A} N\cong_{d-2} B, \quad {\text{and}}\quad
N\otimes_{B} M\cong_{d-2} A$$
as ${\mathbb Z}$-graded bimodules, then the triple $(B,M,N)$ or
simply the algebra $B$ is called a {\it noncommutative
quasi-resolution} (or {\it NQR} for short) of $A$.
\end{definition}

An ungraded version of the above definition is given in Definition
\ref{xxdef3.2} (also see Definition \ref{xxdef3.16} for a related
definition). Note that Van den Bergh considers a normal 
Gorenstein noetherian commutative integral
domain $A$. By a classical theorem of Serre, being normal
 is equivalent to these two conditions: for every prime ideal
$\mathfrak{p}\subseteq A$ of height $\leq 1$ the local ring $A_{\mathfrak{p}}$ is regular,
 and for every prime ideal $\mathfrak{p}\subseteq A$
of height $\geq 2$  the local ring $A_{\mathfrak{p}}$ has depth $\geq 2$,
which is related to Definition \ref{xxdef0.5}.

By Proposition \ref{xxpro7.5}, Van den Bergh's NCCRs
(or Iyama-Reiten's version, see Definition \ref{xxdef0.2}) produce
naturally examples of the ungraded version of NQRs. Noncommutative
examples of NQRs are given in Section 8. Our main theorem is to
prove a version of Conjecture \ref{xxcon0.3} for NQRs in dimension
no more than three.

\begin{theorem}
\label{xxthm0.6}
Fix $\partial$ to be $\GK$ as in the setting of Definition \ref{xxdef0.5}.
Let $A$ be a noetherian locally finite ${\mathbb N}$-graded algebra over
the base field.
\begin{enumerate}
\item[$(1)$]
Suppose $\GK(A)=2$. Then all NQRs of $A$ are Morita equivalent.
\item[$(2)$]
Suppose $\GK(A)=3$. Then all NQRs of $A$ are derived equivalent.
\end{enumerate}
\end{theorem}

The proof of Theorem \ref{xxthm0.6} is given in Section 8. A
version of Theorem \ref{xxthm0.6} holds for other dimension functions
$\partial$ with some extra hypotheses and details are given
in Theorems \ref{xxthm4.2} and \ref{xxthm6.6}. Note that the
hypotheses on $\partial$ (as listed in Theorem \ref{xxthm6.6})
are automatic in the commutative case or the central-finite
case when $\partial=\K$ (see Lemmas \ref{xxlem7.1} and
\ref{xxlem7.2}). Therefore Theorem \ref{xxthm0.6}, or Theorems
\ref{xxthm4.2} and \ref{xxthm6.6} together, generalize
important results of Van den Bergh \cite[Theorem 6.6.3]{VdB2},
Iyama-Reiten \cite[Corollary 8.8]{IR} and Iyama-Wemyss
\cite[Theorem 1.5]{IW1}.

Inspired by the work in \cite{IR,IW1} and Theorem
\ref{xxthm0.6}(1), we have the following question:

\begin{question}
\label{xxque0.7}
Let $B_1$ and $B_2$ be Auslander-regular and $\partial$-$\CM$
algebras with $\gldim B_i=\partial(B_i)=2$ for $i=1,2.$ If
$B_1$ and $B_2$ are derived equivalent, then are they
Morita equivalent?
\end{question}

The paper is organized as follows. Sections 1 and 2 are preliminaries
containing a discussion of dimension functions and homological properties.
A detailed definition and basic properties of a NQR are given in
Section 3. A proof of the main theorem is basically given in Sections
4, 6 and 8, while some technical material is taken care of in Section
5. The connections between NCCRs and NQRs are given in Section 7.
The final section 8 contains examples of NQRs of noncommutative algebras.

\section{Dimension functions and quotient categories}
\label{xxsec1}

Throughout let $\Bbbk$ be a field. All algebras and modules are over
$\Bbbk$. We further assume that {\bf all algebras are noetherian}
in this paper.

We first briefly review background material on dimension functions
and quotient categories of the module categories.

\begin{notation}
\label{xxnot1.1}
For an algebra $A$, we fix the following notations.
\begin{enumerate}
\item[$(1$)]
$\Mod A$ (respectively, $\Mod A^{\op}$): the category of all right
(respectively, left) $A$-modules.
\item[$(2)$]
$\modu A$ (respectively, $\modu A^{\op}$): the full subcategory of
$\Mod A$ (respectively, $\Mod A^{\op}$) consisting of finitely
generated right (respectively, left) $A$-modules.
\item[$(3)$]
$\proj A$ (respectively, $\proj A^{\op}$): the full subcategory of
$\modu A$ (respectively, $\modu A^{\op}$) consisting of finitely generated
projective right (respectively, left) $A$-modules.
\item[$(4)$]
$\refl A$ (respectively, $\refl A^{\op}$): the full subcategory of
$\modu A$ (respectively, $\modu A^{\op}$) consisting of reflexive
right (respectively, left) $A$-modules, see Definition \ref{xxdef2.11}.
\item[$(5)$]
$\add_A(M)$(=$\add M$)  for $M\in\modu A$: the full subcategory
of $\modu A$ consisting of direct summands of finite direct sums
of copies of $M$.
\end{enumerate}
\end{notation}
Usually we work with right modules. We will use the functor
$\Hom_{A}(-,A_A)$ a lot, so let us mention a simple fact below.
A contravariant equivalence between two categories is called a
{\it duality}. Let $A$ be an algebra. Then there is a duality
of categories
$$\Hom_A(-,A_A): \proj A\longrightarrow \proj A^{\op}.$$

Our proof of the main result uses quotient categories
of the module categories defined via a dimension function, so
we first give the following definition, which is a slightly
modification of the definition given in \cite[Section 6.8.4]{MR}.
We also refer to \cite[Section 1]{BHZ1} for a similar definition.

\begin{definition}
\label{xxdef1.2}
A function $\partial:\Mod A\rightarrow \R_{\geq 0}\cup\{\pm\infty\}$ is
called a {\it dimension function} if,
\begin{enumerate}
\item[$(a)$]
$\partial(M)=-\infty$ if and only if $M=0$, and
\item[$(b)$]
for all $A$-modules $M,$
$$\partial(M)\geq\max\{\partial(N),\partial(M/N)\},$$
whenever $N$ is a submodule of $M$.
\end{enumerate}
The $\partial$ is called
an {\it exact dimension function} if, further,
\begin{enumerate}
\item[$(c)$]
for all $A$-modules $M$,
$$\partial(M)=\sup\{\partial(N),\partial(M/N)\},$$
whenever $N$ is any submodule of $M$, and
\item[$(d)$]
for every direct system of submodules of $M$, say $\{M_i\}_{i\in I}$,
\begin{equation}
\label{E1.2.1}\tag{E1.2.1}
\partial (\bigcup_{i\in I} M_i)=\sup \{ \partial(M_i)\mid i\in I\}.
\end{equation}
\end{enumerate}
\end{definition}

Condition (d) in Definition \ref{xxdef1.2} is new. As a consequence
of condition (d), we obtain that, for every $M\in \Mod A$,
\begin{equation}
\label{E1.2.2}\tag{E1.2.2}
\partial(M):=\sup\{ \partial(N) \mid {\text{for all finitely
generated submodules}}\; N\subseteq M\}.
\end{equation}
If we start with an exact dimension function $\partial$ defined on
$\modu A$, then $\partial$ can be extended to $\Mod A$ by using
\eqref{E1.2.2}. In this case, both condition (c) and condition (d)
in Definition \ref{xxdef1.2} are automatic. In fact, natural examples
of dimension functions in this paper are constructed by \eqref{E1.2.2}
from an exact dimension function $\partial$ defined on $\modu A$,
see Remark \ref{xxrem1.4}. One advantage of condition (d) is that,
for any fixed $n$, every right $A$-module $M$ has a maximal submodule
$M'$ with $\partial(M')\leq n$.

Similarly, one can define a dimension function on left modules.
For most of the statements in this paper, we assume the following:

\begin{hypothesis}
\label{xxhyp1.3}
Let $A$ and $B$ be algebras with dimension function $\partial$.
\begin{enumerate}
\item[$(1)$]
$\partial$ is an exact dimension function defined on both right
modules and left modules.
\item[$(2)$]
$\partial(A),\partial(B)\in\N.$
\item[$(3)$]
If an $(A,B)$-bimodule $M$ is finitely generated  as a $B$-module, then
$\partial(_{A}M)\leq\partial(M_B)$. This also holds when
switching $A$ and $B$. In particular, for an $(A,B)$-bimodule
$M$ which is finitely generated  both as a left $A$-module and as a right
$B$-module, we have
\begin{equation}
\label{E1.3.1}\tag{E1.3.1}
\partial(_{A}M)=\partial(M_B).
\end{equation}
A dimension function $\partial$ is called {\it symmetric} if
\eqref{E1.3.1} holds, which is identical to \cite[Definition 2.20]{YZ2}.
\end{enumerate}
\end{hypothesis}

Note that symmetry condition \eqref{E1.3.1} resembles 
``symmetric derived torsion'', in the sense of \cite[Section 9]{VY}.

Unless otherwise stated, an $(A,B)$-bimodule means finitely generated
on both sides. Recall that $A$ is called {\it central-finite} if it is
a finitely generated module over its center.

\begin{remark}
\label{xxrem1.4}
Two standard choices of $\partial$ are the Gelfand-Kirillov
dimension, denoted by $\GK$, see \cite[Ch.8]{MR} and
\cite{KL}, and the Krull dimension, denoted by $\K$, see
\cite[Ch.6]{MR}. As a convention, we define the Krull
dimension of an infinitely generated module via \eqref{E1.2.2}.
Note that, while $\K$ is defined for every
noetherian ring $A$, it is not known whether it is always
symmetric. If $A$ is central-finite, then $\K$ is symmetric
\cite[Corollary 6.4.13]{MR}. On the other hand, $\GK$ is
always symmetric \cite[Corollary 5.4]{KL}, though it could
be infinite for a nice noetherian $\Bbbk$-algebra. For a
central-finite algebra $A$ with affine center, $\K$ coincides
with $\GK;$ this is an easy consequence of the equality of
the two dimensions for affine commutative algebras
\cite[Theorem 4.5]{KL}.
\end{remark}

We need to recall some definitions and notations introduced
in \cite{BHZ1}. From now on, we fix an exact dimension
function, say $\partial$. We use $n$ for a nonnegative integer.
Let $\Mod_n A$ denote the full subcategory of $\Mod A$
consisting of right $A$-modules $M$ with $\partial(M)\leq n$.
Since $\partial$ is exact, $\Mod_n A$ is a Serre subcategory
of $\Mod A$. Hence it makes sense to define the quotient categories:
$$\QMod_nA:=\frac{\Mod A}{\Mod_nA}, \quad \text{and} \quad
\qmod_nA:=\frac{\modu A}{\modu_nA},$$
which can be seen as a generalized noncommutative scheme (see \cite{AZ}).
Note that there are some symmetries between $\qmod_nA$ and
$\qmod_nA^{\op}$ when $A$ admits a nice dualizing complex, see 
\cite[Theorem 2.15]{YZ2}. We denote the natural and exact 
projection functor by
\begin{equation}
\label{E1.4.1}\tag{E1.4.1}
\pi:\Mod A\longrightarrow\QMod_n A.
\end{equation}
For $M\in\Mod A,$ we write $\mathcal{M}$ for the object $\pi(M)$
in $\QMod_n A$. The hom-set in the quotient category is defined by
\begin{equation}
\label{E1.4.2}\tag{E1.4.2}
\Hom_{\QMod_n A}(\mathcal{M},\mathcal{N})=
\lim_{\longrightarrow}\Hom_A(M',N')
\end{equation}
for $M$, $N$ $\in\Mod A$, where $M'$ is a submodule of $M$ such
that $\partial(M/M')\leq n$, $N'=N/T$ for some submodule $T\subseteq N$
with $\partial(T)\leq n$, and where the direct limit runs over
all the pairs $(M',N')$ with these properties.

\begin{definition}
\label{xxdef1.5}
Let $n\geq 0$. Let $A$ and $D$ be two algebras.
\begin{enumerate}
\item[$(1)$]
Two right $A$-modules $X,Y$ are called
{\it $n$-isomorphic}, denoted by $X\cong_n Y$, if there
exist a right $A$-module $P$ and morphisms $f: X\to P$
and $g: Y\to P$ such that both the kernel and cokernel of
$f$ and $g$ are in $\Mod_n A$.
\item[$(2)$]
Two $(D,A)$-bimodules $X,Y$ are called
{\it $n$-isomorphic}, denoted by $X\cong_n Y$, if there
exist a $(D,A)$-bimodule $P$ and bimodule morphisms $f: X\to P$
and $g: Y\to P$ such that both the kernel and cokernel of
$f$ and $g$ are in $\Mod_n A$ when viewed as right $A$-modules.
\end{enumerate}
\end{definition}

Definition \ref{xxdef1.5}(2) is useful when we consider
bimodules. Most of right module statements regarding
$n$-isomorphisms have bimodule analogues, for which we might
omit the proofs if these are clear.

\begin{remark}
\label{xxrem1.6}
\begin{enumerate}
\item[(1)]
Since we usually consider finitely generated modules,
it turns out that we are mostly talking about $n$-isomorphisms
in $\modu A$. In this case, we can take $P\in \modu A$
in the above definition.
\item[(2)]
If $X$ is a submodule of $Y$ in $\modu A$ and $Y/X\in\modu_n A$,
then clearly $X\cong_n Y.$
\item[(3)]
If
$$0 \rightarrow K\rightarrow M\rightarrow N\rightarrow C \rightarrow 0$$
is an exact sequence in $\modu A$ with $K,C\in\modu_n A$, then
$M\cong_n N$ in $\modu A.$
\end{enumerate}
\end{remark}

The following lemma is easy and the proof is omitted.

\begin{lemma}
\label{xxlem1.7}
Two right $A$-modules $X$ and $Y$ are $n$-isomorphic in $\modu A$
if and only if their images ${\mathcal X}$ and ${\mathcal Y}$
in $\qmod_n A$ are isomorphic.
\end{lemma}


\begin{definition}\cite[Definition 1.2]{BHZ1}
\label{xxdef1.8}
Let $A$ and $B$ be algebras and $\partial$ be an exact dimension
function that is defined on $A$-modules and  $B$-modules. Let
$n$ and $i$ be nonnegative integers. Suppose that $_{A}M_B$ is a
bimodule.
\begin{enumerate}
\item[(1)]
We say $\partial$ satisfies $\gamma_{n,i}(M)^l$ if for any
$N\in\modu_n A$, $\Tor_j^A(N,M)\in\modu_n B$ for all $0\leq j\leq i$.
\item[(2)]
We say $\partial$ satisfies $\gamma_{n,i}(M)^r$ if for any
$N\in\modu_n B^{\op}$, $\Tor_j^B(M,N)\in\modu_n A^{\op}$ for all
$0\leq j\leq i$.
\item[(3)]
We say $\partial$ satisfies $\gamma_{n,i}(M)$ if it satisfies
$\gamma_{n,i}(M)^l$ and $\gamma_{n,i}(M)^r.$
\item[(4)]
We say $\partial$ satisfies $\gamma_{n,i}(A,B)^l$ if it satisfies
$\gamma_{n,i}(M)^l$ for all $_{A}M_B$ that are finitely generated
on both sides.
\item[(5)]
We say $\partial$ satisfies $\gamma_{n,i}(A,B)^r$ if it satisfies
$\gamma_{n,i}(M)^r$ for all $_{A}M_B$ that are finitely generated
on both sides.
\item[(6)]
We say $\partial$ satisfies $\gamma_{n,i}(A,B)$ if it satisfies
$\gamma_{n,i}(M)$ for all $_{A}M_B$ that are finitely generated
on both sides.
\end{enumerate}
\end{definition}

Note that the conditions listed in the above definition are related 
to some conditions in terms of symmetric (derived) torsion, 
generalizing the $\chi$-condition of \cite{AZ}, see 
\cite[Section 16.5]{Y2}.


In the most parts of this paper, we will be particularly interested
in the $\gamma_{n,1}$ property.

\begin{lemma}\cite[Lemma 1.3]{BHZ1}
\label{xxlem1.9}
Let $A$ and $B$ be algebras such that $\partial$ is an exact
dimension function on $A$-modules and  $B$-modules. Assume
that $\partial$ satisfies $\gamma_{n,1}(M)^l$ for a bimodule
$_{A}M_B.$ Then the functor $-\otimes_AM$ induces a functor
$$-\otimes_{\mathcal{A}}\mathcal{M}:\QMod_nA\longrightarrow\QMod_nB.$$
Since $M$ is finitely generated on both sides, this functor restricts to:
$$-\otimes_{\mathcal{A}}\mathcal{M}:\qmod_nA\longrightarrow\qmod_nB.$$
\end{lemma}


\begin{lemma}
\label{xxlem1.10}
Retain the hypotheses in Lemma \ref{xxlem1.9}.
Suppose that $X$ and $Y$ are in $\modu A$ such that
$X\cong_n Y$. Then $X\otimes_A M\cong_n Y\otimes_A M$ in $\modu B$.
\end{lemma}

\begin{proof} The assertion follows from Lemmas \ref{xxlem1.7}
and \ref{xxlem1.9} or the proof of \cite[Lemma 1.3]{BHZ1}.
\end{proof}

We will also use the right adjoint functor of $\pi$ defined in
\eqref{E1.4.1}. Since $\Mod_n A$ is a Serre subcategory
(or a {\it dense} subcategory in the sense of \cite[Sect.4.3]{Po}),
every right $A$-module has a largest submodule in $\Mod_n A$
(see also \eqref{E1.2.1}). Note that $\Mod A$ is {\it locally small}
(in the sense of \cite[p. 5]{Po}) and has enough injective objects.
By a well-known classical category theory result
\cite[Theorem 4.4.5 or Proposition 4.5.2]{Po} (which is in a
different mathematical language unfortunately), there
is a {\it section functor}, denoted by $\omega$ (we are
following the notation of \cite[p. 234]{AZ}) such that there
is a natural isomorphism
\begin{equation}
\label{E1.10.1}\tag{E1.10.1}
\Hom_{\Mod A}(N,\omega ({\mathcal M}))\cong
\Hom_{\QMod_n A}(\pi(N), {\mathcal M}),
\end{equation}
for all $M\in \Mod A$ and ${\mathcal M}\in \QMod_n A$.
Given a module $M$, let $C_M$ denote the filtering
category of maps $M\to M'$ whose kernel and cokernel
are in $\Mod_n A$. Then
\begin{equation}
\label{E1.10.2}\tag{E1.10.2}
\omega \pi(M)=\lim_{(M\to M')\in C_M} M'.
\end{equation}
By \cite[Proposition 4.4.3]{Po}, the unit of the adjunction
$Id\to \omega\pi$ induces the natural map
\begin{equation}
\label{E1.10.3}\tag{E1.10.3}
u_M: M\to \omega\pi (M),
\end{equation}
which has kernel and cokernel in $\Mod_n A$. Further, $u_M$
is an isomorphism if and only if $M$ is {\it closed}
in the sense of \cite[P. 176]{Po}. By
\cite[Lemma 4.4.6(2)]{Po}, the image of $u_M$, which is
canonically isomorphic to $M/M'$ where $M'$ is the
largest subobject of $M$ in $\Mod_n A$, is an essential
subobject in $\omega\pi (M)$. The assertions in the
following lemma are known to experts.

\begin{lemma}
\label{xxlem1.11}
Let $A$ and $B$ be two algebras and $n$ be a nonnegative integer.
Let $M$ be a right $A$-module and $M'$ be the largest submodule
of $M$ such that $\partial(M')\leq n$.
\begin{enumerate}
\item[(1)]
$\omega\pi(M)$ is naturally isomorphic to the largest submodule
of the injective hull of $M/M'$ containing $M/M'$ such
that $X/(M/M')$ is in $\Mod_n A$.
\item[(2)]
If $M$ is a $(B,A)$-bimodule, then
$\omega\pi(M)$ is a $(B,A)$-bimodule and
$u_M$ is a bimodule morphism.
\end{enumerate}
\end{lemma}

\begin{proof}
(1) Without loss of generality, we can assume that
$M$ does not contain a nonzero submodule of $\partial$-dimension
$\leq n$, namely, $M'=0$. Let $C(M)$ be the largest submodule $X\supseteq M$
of the injective hull of $M$ such that $X/M$ is in $\Mod_n A$.
By \cite[Lemma 4.4.6(2)]{Po}, we have canonical injective maps
$$M\xrightarrow{u_M} \omega\pi(M)\xrightarrow{f} C(M)$$
such that the cokernel of $f$ is in $\Mod_n A$. In
particular, $\pi(f)$ is an isomorphism. Applying the
natural transformation $Id\to \omega\pi$, we have a
commutative diagram
$$\begin{CD}
\omega\pi(M) @> f >> C(M)\\
@V u_{\omega\pi(M)} VV  @VV u_{C(M)}V\\
\omega\pi\omega\pi(M) @>> \omega\pi(f) >\omega\pi(C(M)).
\end{CD}
$$
Since $\pi\omega\cong Id$, $u_{\omega\pi(M)}$ is an isomorphism.
Since $\pi(f)$ is an isomorphism, $\omega\pi(f)$ is an isomorphism.
Note that both $f$ and $u_{C(M)}$ are injective. Hence
$f$ and $u_{C(M)}$ are isomorphisms.

(2) Since $M$ is a $(B,A)$-bimodule, there is an algebra map
$B\to \End_{\Mod A}(M_A)$. Applying the functor $\omega\pi$, we
obtain an algebra map
$$B\to \End_{\Mod A}(M_A)\to \End_{\Mod A}(\omega\pi(M)),$$
which means that $\omega\pi(M)$ is a $(B,A)$-bimodule.
Since the unit of the adjunction $Id\to \omega\pi$ is a natural
transformation, $u_M$ is a bimodule morphism.
\end{proof}

\section {Preliminaries on homological properties}
\label{xxsec2}

In this section, we review some homological properties that are needed
in the definition of a noncommutative quasi-resolution.

\begin{definition}\cite[Definitions 1.2, 2.1, 2.4]{Le}
\label{xxdef2.1}
Let $A$ be an algebra and $M$ a right $A$-module.
\begin{enumerate}
\item[(1)]
The {\it grade number} of $M$ is defined to be
$$j_A(M):=\inf\{i|\Ext_A^i(M,A)\neq0\}\in \N\cup\{+\infty\}.$$
If no confusion can arise, we write $j(M)$ for $j_A(M)$. Note that
$j_A(0)=+\infty$.
\item[(2)]
A nonzero $A$-module $M$ is called {\it $n$-pure} (or just
{\it pure}) if $j_A(N)=n$ for all nonzero finitely generated
submodules $N$ of $M$.
\item[(3)]
We say $M$ satisfies the {\it Auslander condition} if for any $q\geq0,$
$j_A(N)\geq q$ for all left $A$-submodules $N$ of $\Ext_A^q(M,A)$.
\item[(4)]
We say $A$ is {\it Auslander-Gorenstein} (respectively,
{\it Auslander regular}) of dimension $n$ if
$\injdim A_A=\injdim {_AA}=n<\infty$ (respectively, $\gldim A=n<\infty$)
and every finitely generated left and right $A$-module satisfies
the Auslander condition.
\end{enumerate}
\end{definition}

\begin{proposition}\cite[Proposition 1.8]{Bj}
\label{xxpro2.2}
Let $A$ be Auslander-Gorenstein. If
$$0\longrightarrow M'\longrightarrow M\longrightarrow M''\longrightarrow0$$
is an exact sequence of finitely generated $A$-modules, then
$$j(M)=\inf\{j(M'),j(M'')\}.$$
\end{proposition}

\begin{definition}\cite[Definition 0.4]{ASZ1}
\label{xxdef2.3}
Let $A$ be an algebra with a dimension function $\partial$.
We say $A$ is {\it $\partial$-Cohen-Macaulay} (or,
{\it $\partial$-$\CM$}
in short) if $\partial(A)=d\in\N,$ and
$$j(M)+\partial(M)=\partial(A)$$
for every finitely generated nonzero left (or right) $A$-module $M$.
If $A$ is $\GK$-Cohen-Macaulay, namely, if $\partial=\GK$, we
just say it is {\it Cohen-Macaulay} or {\it $\CM$}.
\end{definition}

There are other modified definitions of noncommutative 
CM algebras, in particular, using the rigid Auslander dualizing 
complex over the algebra $A$ \cite{YZ2}. Here we are using a 
more classical approach in Definition \ref{xxdef2.3}. 

\begin{remark}
\label{xxrem2.4}
\begin{enumerate}
\item[(1)]
If $A$ is $\partial$-$\CM$ with $\partial(A)\in\N$, then
$j_A(M)<\infty$ and $\partial(M)\in\N$ for all nonzero
finitely generated $A$-modules $M$.
\item[(2)]
If $A$ is a $\partial$-$\CM$ algebra, then $\partial$ is
an exact dimension \cite[p. 3]{ASZ2}. In particular,
$\GK$ is exact on finitely generated modules over $\CM$ algebras.
\item[(3)]
Let $A$ be Auslander-Gorenstein. The {\it canonical dimension}
of a finitely generated right (or left) $A$-module $M$
is defined to be
\begin{equation}
\label{E2.4.1}\tag{E2.4.1}
\partial (M)=\injdim A-j(M),
\end{equation}
which was introduced in \cite[Definition 2.9]{YZ2}, 
in the more general setting of Auslander dualizing complexes.
For infinitely generated modules, see \eqref{E1.2.2}.
By \cite[Proposition 4.5]{Le}, the canonical dimension is an
exact (but not necessarily symmetric) dimension function.
By \eqref{E2.4.1}, $A$ is trivially $\partial$-CM.
\end{enumerate}
\end{remark}

\begin{definition} \cite[Definition 1.12]{Bj}
\label{xxdef2.5}
Let $M$ be a finitely generated pure right $A$-module, see
Definition \ref{xxdef2.1}(2). A {\it tame and pure} extension
of $M$ is a finitely generated right $A$-module
$N$ such that $M\subseteq N$, $N$ is pure and
$j(N/M)\geq j(M)+2$. Note that a tame and pure extension
is always an essential extension.
\end{definition}

The following result of Bj{\"o}rk is called
{\it Gabber's Maximality Principle}, see
\cite[Theorem 1.14]{Bj}.

\begin{theorem}
\cite[Theorem 1.14]{Bj}
\label{xxthm2.6}
Let $A$ be an Auslander-Gorenstein algebra. Suppose that $M$
is a finitely generated $n$-pure $A$-module. Let $N$ be an
$A$-module containing $M$ such that every nonzero finitely
generated submodule of $N$ is $n$-pure. Then $N$ contains
a unique largest tame and pure extension of $M$.
\end{theorem}

We do not assume that $N$ is finitely generated in the above
theorem. On the other hand, by definition, a tame and pure
extension of $M$ is finitely generated. We will explain the
Gabber's Maximality Principle in some details in the following
two lemmas. Firstly, we recall some functors. Let $\partial$ be the
canonical dimension defined in Remark \ref{xxrem2.4}(3) when $M$
is finitely generated and extended to $\Mod A$ by \eqref{E1.2.2}.
We fix a non-negative integer $n$ and let $d=\injdim A$. If $M$
is $n$-pure, then $\partial(M)=d-n$ by \eqref{E2.4.1}. In the
next two lemmas, $M$ will be an $n$-pure right $A$-module. Let
$$\pi: \Mod A\to \QMod_{d-n-2} A$$
and
$$\omega: \QMod_{d-n-2} A\to \Mod A,$$
see \eqref{E1.4.1} and \eqref{E1.10.1}.

The following lemma is a special case of \cite[Theorem 2.19]{YZ2}.
For the convenience of readers, we give detailed proof. 

\begin{lemma}
\label{xxlem2.7}
Let $A$ be an Auslander-Gorenstein algebra. Suppose that $M$ is
a finitely generated $n$-pure $A$-module. Then there is an
$n$-pure $A$-module $\widetilde{M}$, unique up to unique
isomorphism, such that the following hold.
\begin{enumerate}
\item[$(1)$]
$\widetilde{M}$ is a tame and pure extension of $M$, namely,
there is a given injective morphism $g_{M}: M\to \widetilde{M}$,
\item[$(2)$]
If $N$ is a tame and pure extension of $M$, then $g_M$ factors
uniquely through the inclusion map $M\to N$.
\end{enumerate}
Further, $\widetilde{M}$ is naturally isomorphic to both
$\omega\pi(M)$ and $\Ext^n_{A^{\op}}(\Ext^n_A(M,A),A)$ and
$g_M$ agrees with $u_M$ in \eqref{E1.10.3} when $\widetilde{M}$
is identified with $\omega\pi(M)$.
\end{lemma}

\begin{proof} Let $\partial$ be the canonical dimension
defined by \eqref{E2.4.1} and $d=\injdim A$.

(1) Let $E(M)$ be the injective hull of $M$. By
Lemma \ref{xxlem1.11}(1), $\omega\pi(M)$ is a largest
submodule of $E(M)$ containing $M$ such that $\omega\pi(M)/M$
has $\partial$-dimension at most $d-n-2$. So every nonzero
finitely generated submodule $N(\supseteq M)$ of $\omega\pi(M)$
is $n$-pure and $j(N/M)\geq n+2$. Thus $N$ is a tame and
pure extension of $M$. By Theorem \ref{xxthm2.6}, $\omega\pi(M)$
contains a largest (and maximal) tame and pure extension, which
must be $\omega\pi(M)$ itself. So $\omega\pi(M)$ satisfies (1).
We now define $\widetilde{M}=\omega\pi(M)$ and $g_M$ to be
the inclusion map.

(2) Let $N$ be a tame and pure extension of $M$. Then $N$ is an
essential extension of $M$. Let $i: M\to N$ be the inclusion
map. Since  $E(M)$ is injective, there is an injective
map $f: N\to E(M)$ such that $f\circ i: M\to E(M)$ is the inclusion
map. Since $N$ is a tame and pure extension of $M$, it is easy to
see that the image of $f$ is inside $\omega\pi(M)$. Thus
we have a map $f: N\to \widetilde{M}:=\omega\pi(M)$ such that
$g_{M}=f\circ i$. Finally we prove the uniqueness of this
factorization. Suppose there are two maps $f_1, f_2$ such that
$g_{M}=f_1\circ i=f_2\circ i$. Then $(f_1-f_2)\circ i=0$ or
the $(f_1-f_2)(M)=0$. Then the image of $f_1-f_2$ is a
quotient module of $N/M$, which has $\partial$-dimension strictly
less than $d-n$. Since $\widetilde{M}$ is $n$-pure, $f_1-f_2$
must be zero, namely, $f_1=f_2$. This shows that uniqueness.

Part (2) can be considered as a universal property.
The uniqueness of $\widetilde{M}$ follows from part (2).

For the last assertion, we let $M^{\ast\ast}:=
\Ext^n_{A^{\op}}(\Ext^n_A(M,A),A)$.
By \cite[Lemma 2.2]{Le} and \cite[Proposition 1.13]{Bj},
$M^{\ast\ast}$ is a
tame and pure extension and there is no other tame and pure
extensions properly containing $M^{\ast\ast}$. Therefore
$M^{\ast\ast}\cong \omega\pi(M)$ by part (2).
\end{proof}

\begin{definition}
\label{xxdef2.8}
Let $A$ be an Auslander-Gorenstein algebra. Suppose that $M$ is
a finitely generated $n$-pure right $A$-module. The map
$g_M: M\to \widetilde{M}$ (or simply the module $\widetilde{M}$)
in Lemma \ref{xxlem2.7}, is called a {\it Gabber closure} of $M$.
By Lemma \ref{xxlem2.7}, a Gabber closure of $M$ always exists
and is unique up to a unique isomorphism. Therefore, it is no
confusion to call it {\it the Gabber closure} of $M$. In this
case, we write the Gabber closure as $g_M: M\to G_A(M)$ (or
simply $G_A(M)$).
\end{definition}

Suppose $\partial$ is an arbitrary dimension function.
When $A$ is a $\partial$-$\CM$ algebra, $\partial$ equals to
the canonical dimension up to a uniform shift. Hence
\eqref{E2.4.1} implies that the condition
$$j(\widetilde{M}/M)\geq j(M)+2$$
is equivalent to
$$\partial(\widetilde{M}/M)\leq\partial(M)-2.$$

\begin{lemma}
\label{xxlem2.9}
Let $A$ be an Auslander-Gorenstein algebra. Suppose that
$M$ is a finitely generated $n$-pure right $A$-module. Let
$N$ be an $n$-pure $A$-module such that
\begin{enumerate}
\item[$(a)$]
$N$ is an essential extension of $M$, and
\item[$(b)$]
$j(N'/M)\geq j(M)+2$ for all finitely generated $A$-submodule
$N'$ of $N$ that contains $M$.
\end{enumerate}
Then $N$ is a finitely generated $A$-module.
\end{lemma}

\begin{proof}
In this proof,  let $M^{**}$ denote $\Ext^n_{A^{\op}}(\Ext^n_A(M,A),A)$.
If $N$ is not finitely generated, then there is an ascending
chain of finitely generated $A$-submodules
$$M\subsetneq M_1\subsetneq M_2\subsetneq\cdots\subsetneq N$$
such that $M_i\in\modu A$ are $n$-pure and
$j(M_i/M)\geq j(M)+2$ for every $i$. By \cite[Lemma 1.15]{Bj},
$M_i^{**}= M^{**}$. Moreover, since every
$M_i$ are $n$-pure, we have
$$M\subseteq M_1\subseteq M_2\subseteq\cdots\subseteq\cdots
\subseteq M_i^{**}=M^{**}$$
for every $i$. Since $M^{**}$ is finitely generated by Lemma \ref{xxlem2.7},
the ascending  chain stabilizes, a contradiction.
\end{proof}

We collect some facts and re-statements concerning the Gabber closure.

\begin{proposition}
\label{xxpro2.10}
Let $A$ be an Auslander-Gorenstein algebra. Suppose that
$M$ is a finitely generated $n$-pure $A$-module.
\begin{enumerate}
\item[(1)]
The Gabber closure of $M$, denoted by $g_M: M\to G_A(M)$ as in
Definition \ref{xxdef2.8}, exists and is unique up to a unique
isomorphism.
\item[(2)]
$g_M$ agrees with $u_M$ in \eqref{E1.10.3} for specific choices
of $\pi$ and $\omega$ given before Lemma \ref{xxlem2.7}.
\item[(3)]
$G_A(M)$ is a tame and pure extension of $M$. In particular,
$G_A(M)$ is finitely generated over $A$.
\item[(4)]\cite[Proposition 1.13]{Bj}
$G_A(M)$ does not have any proper tame and pure extension.
\item[(5)]
Let $N$ be a tame and pure extension of $M$. If $N$ does not
have any proper tame and pure extension, then $N\cong G_A(M)$.
\item[(6)]
If $M$ is a $(B,A)$-bimodule, then $G_A(M)$ is a $(B,A)$-bimodule
and $g_M$ is a morphism of $(B,A)$-bimodules.
\end{enumerate}
\end{proposition}

\begin{proof}
(1,2,3) See Lemma \ref{xxlem2.7}.

(4) Since $G_A(M)$ is identified with
$\Ext^n_{A^{\op}}(\Ext^n_A(M,A),A)$, the assertion is exactly
\cite[Proposition 1.13]{Bj}.

(5) By Lemma \ref{xxlem2.7}(2), $G_A(M)$ is a tame and pure
extension of $N$. Since $N$ does not have a proper
tame and pure extension, $N=G_A(M)$.

(6) Since $G_A(M)$ can be identified with $\omega\pi(M)$,
the assertion follows from Lemma \ref{xxlem1.11}(2).
\end{proof}

For every right $A$-module $M$, let
\begin{equation}
\label{E2.10.1}\tag{E2.10.1}
M^{\vee}=\Hom_A(M,A)
\end{equation}
and
\begin{equation}
\label{E2.10.2}\tag{E2.10.2}
M^{\vee\vee}=\Hom_{A^{\op}}(\Hom_A(M,A),A).
\end{equation}
When $n=0$ as in the proof of Lemma \ref{xxlem2.9},
$M^{\ast\ast}=M^{\vee\vee}$. By adjunction, there
is a natural map $M\longrightarrow M^{\vee\vee}:=
\Hom_{A^{\op}}(\Hom_A(M,A),A)$.

\begin{definition}
\label{xxdef2.11}
Let $A$ be an algebra.
A finitely generated right $A$-module $M$ is called
{\it reflexive} if the natural morphism
$M\longrightarrow M^{\vee\vee}$ is an isomorphism.
A reflexive left module is defined similarly.
\end{definition}

It's obvious that when $A$ is Auslander-Gorenstein,
an $A$-module $M$ of maximal dimension is reflexive 
if and only if $M$ is its own Gabber closure. 
Note that the definition of a reflexive module given in
\cite{IR, IW1, IW2} is relative to a given base commutative ring.
It is clear that every projective module is reflexive, but the
converse is not true. The following lemma and corollary are well-known.

\begin{lemma}
\label{xxlem2.12}
Let $A$ be an algebra of global dimension $d$.
If $M\in\modu A,$ then $\pd_{A^{\op}}M^{\vee}\leq \max\{0,d-2\}$,
where $(-)^{\vee}=\Hom_A(-,A).$
\end{lemma}

\begin{proof}
Suppose that $\cdots\longrightarrow P_1\longrightarrow P_0
\longrightarrow M\longrightarrow0$ is a projective resolution of $M$
such that each $P_i$ is finitely generated. Applying $(-)^{\vee}$
to the above exact sequence, there is an exact
sequence of left $A$-modules
$$0\longrightarrow M^{\vee}\longrightarrow P_0^{\vee}
\longrightarrow P_1^{\vee}\longrightarrow E\longrightarrow0,$$
where $E=\coker(P_0^{\vee}\rightarrow P_1^{\vee})$. Since
$\pd_{A^{\op}}E\leq d$ and $P_0^{\vee},P_1^{\vee}\in\proj A^{\op}$,
we have $\pd_{A^{\op}} M^{\vee}\leq \max\{0,d-2\}$.
\end{proof}

\begin{corollary}
\label{xxcor2.13}
Let $A$ be an algebra of global dimension $d$. If $ M\in\refl A$,
then $\pd_A M\leq\max\{0, d-2\}$. In particular, if $d\leq 2$,
then any reflexive $A$-module is projective.
\end{corollary}

\begin{proof} Use Lemma \ref{xxlem2.12} and the fact
$M\cong\Hom_{A^{\op}}(\Hom_A(M,A),A)$.
\end{proof}

Next we recall some results about spectral sequences.


If $A$ is noetherian with $\injdim A<\infty$ and $M$ is a finitely
generated $A$-module, then there is a convergent spectral sequence
\cite[Theorem 2.2(a)]{Le}, see \eqref{E2.13.1} below. To simplify
notation later, we use a non-standard indexing of $E_2^{pq}$, with
our indexing, the boundary maps on the $E_2$-page are
$d_2^{p,q}: E_2^{pq}\rightarrow E_2^{p+2,q+1}$:
\begin{equation}
\label{E2.13.1}\tag{E2.13.1}
E_2^{pq}:=\Ext_{A^{\op}}^p(\Ext_A^q(M,A),A)\Rightarrow \Hm^{p-q}(M)
:=\left\{
\begin{array}{ll}
  0, & \mbox{if $p\neq q$,}\\
  M, &  \mbox{if $p=q$.}
\end{array}
\right.
\end{equation}
When $A$ is Auslander-Gorenstein
with $\injdim A=d$, there is a canonical filtration
\begin{equation}
\label{E2.13.2}\tag{E2.13.2}
0=F^{d+1}M\subseteq F^dM\subseteq\cdots\subseteq F^1M\subseteq F^0M=M
\end{equation}
such that $F^pM/F^{p+1}M\cong E_{\infty}^{pp}$. By \cite[Theorem 2.2]{Le},
for each $p$, there exists an exact sequence
$$0\longrightarrow E_{\infty}^{pp}\longrightarrow E_2^{pp}
\longrightarrow Q(p)\longrightarrow0$$
with $j(Q(p))\geq p+2.$


We collect some facts which can be shown by using the above spectral
sequences.

\begin{proposition}\cite[Theorem 2.4]{Le}
\label{xxpro2.14}
Let $A$ be Auslander-Gorenstein and $M$ be a nonzero finitely
generated $A$-module. If $n=j_A(M)$, then $\Ext_A^n(M,A)$
is $n$-pure and $(E_2^{pp}=)\Ext_{A^{\op}}^p(\Ext_A^p(M,A),A)$
is either $0$ or $p$-pure for every integer $p$.
\end{proposition}

\begin{proposition}\cite[Proposition 1.9]{Bj}
\label{xxpro2.15}
Let $A$ be Auslander-Gorenstein. Then a finitely generated $A$-module
$M$ is $j(M)$-pure if and only if $E_2^{pp}=0$ for any $p\neq j(M).$
\end{proposition}

\begin{corollary}
\label{xxcor2.16}
Let $A$ be Auslander-Gorenstein and $M$  a nonzero reflexive $A$-module.
Then $M$ is $0$-pure, and $E_2^{pp}=0$ for any $p\neq 0$. As a consequence,
if $A$  is also a $\partial$-$\CM$ algebra, then $\partial(M)=\partial(A)=
\partial(N)$ for any nonzero reflexive $A$-module $M$ and any nonzero
submodule $N$ of $M$.
\end{corollary}

\begin{proof}
Suppose that $M$ is a nonzero reflexive $A$-module, then $j(M)=0$ and
$$0\neq E_2^{00}=\Hom_{A^{\op}}(\Hom_A(M,A),A)\cong M,$$
which  is $0$-pure by Proposition \ref{xxpro2.14}. By Proposition
\ref{xxpro2.15}, $E_2^{pp}=0$ for every $p\neq 0$. The remaining statement
follows by the definition of $\partial$-CM.
\end{proof}

\begin{lemma}
\label{xxlem2.17}
Let $A$ be Auslander-Gorenstein and $M$ a finitely generated $m$-pure
$A$-module where $m=j(M)$. Then $M=F^mM\supseteq F^{m+1}M=0.$ Further,
$M=E_{\infty}^{mm}M\subseteq E_2^{mm}M.$ In particular,
if $M$ is a $0$-pure module, then
$$M\subseteq M^{\vee\vee}:=\Hom_{A^{\op}}(\Hom_A(M,A),A).$$
\end{lemma}

\begin{proof}
If $M$ is $m$-pure, then $E_2^{pp}=0$ for every $p\neq m$. Therefore
$$E_{\infty}^{pp}=0=F^pM/F^{p+1}M$$ for $p\neq m$. Taking $p=m+1$, we
obtain that $F^{m+1}M=F^{m+2}M=\cdots=0$, as required.
\end{proof}

\begin{proposition}
\label{xxpro2.18}
Let $A$ be an Auslander regular algebra with $\gldim A=3$. If $M$ is
a nonzero reflexive $A$-module, then
$$E_2^{pq}\cong\left\{
\begin{array}{ll}
  M, & \mbox{if $p=q=0$},\\
  E_2^{10}\cong E_2^{31}, & \mbox{if $p=1,q=0$},\\
  E_2^{31} \cong E_2^{10}, & \mbox{if $p=3,q=1$},\\
  0,& \mbox{otherwise}.
\end{array}
\right.$$
\end{proposition}

\begin{proof}
Since  $A$ is Auslander regular,  we have that the $E_2$ table
for $M$  looks like
$$
\begin{array}{cccc}
0& 0& 0& E^{33}\\
0& 0& E^{22}& E^{32}\\
0& E^{11}& E^{21}& E^{31}\\
E^{00}& E^{10}& E^{20}& E^{30}
\end{array}
$$
with $E^{20}=E^{30}=0$ by Lemma \ref{xxlem2.12},
$E^{11}=E^{22}=E^{33}=0$ by Propositions \ref{xxpro2.14} and
\ref{xxpro2.15}, and $E^{32}=0$ since $\pd M\leq 1$ by
Lemma \ref{xxlem2.12}. Hence the $E_2$ table reduces to
$$
\begin{array}{cccc}
0& 0& 0& 0\\
0& 0& 0& 0\\
0& 0& E^{21}& E^{31}\\
E^{00}& E^{10}& 0& 0
\end{array}
$$
and so it suffices to show that $E^{21}=0.$
By \eqref{E2.13.1}-\eqref{E2.13.2}, there is
a canonical filtration
$0=F^4M\subseteq F^3M\subseteq F^2M\subseteq F^1M\subseteq F^0M=M$
such that $F^pM/F^{p+1}M\cong E_{\infty}^{pp}$. It is obvious that
$F^1M=0$. Then $E_{\infty}^{pp}=0$ for every $p\neq0$ by Proposition
\ref{xxpro2.15}, $E_{\infty}^{00}\cong F^0M/F^1M=M$, and further,
$d_2^{00}=0$ (as $M$ being reflexive), which implies that $E^{21}=0$.
Thus, the $E_2$ table for $M$ now looks like
$$
\begin{array}{cccc}
0& 0& 0& 0\\
0& 0& 0& 0\\
0& 0& 0& E^{31}\\
E^{00}& E^{10}& 0& 0
\end{array}
$$
with $E^{10}\cong E^{31}$. The assertion follows.
\end{proof}

\begin{lemma}
\label{xxlem2.19}
Let $A$ be an Auslander Gorenstein algebra and
$M$ be a finitely generated right $A$-module.
\begin{enumerate}
\item[(1)]
$M^{\vee}$ is either 0 or a finitely generated
reflexive left $A$-module.
\item[(2)]
If $M$ is 0-pure, then the Gabber closure $G_A(M)$ is
reflexive.
\end{enumerate}
\end{lemma}

\begin{proof} (1) It is well-known that $M^\vee$
is a finitely generated left $A$-module.

Let $N$ be the largest submodule of $M$ such that
$j(N)>0$ or $N^{\vee}=0$. Then we have a short
exact sequence
$$0\to N\to M\to M/N\to 0,$$
which gives rise to an exact sequence
$$0\to (M/N)^{\vee}\to M^{\vee}\to N^\vee\to \cdots .$$
Since $N^{\vee}=0$, we have $(M/N)^{\vee}\cong
M^{\vee}$. To prove the assertion one may assume
that $N=0$, or equivalently, $M$ is 0-pure.

By \cite[Proposition 1.13]{Bj} (also see Lemma
\ref{xxlem2.7}), there is a short exact sequence
$$0\to M\to M^{\vee\vee}\to M'\to 0,$$
where $j(M')\geq 2$. The above short exact sequence
gives rise to an exact sequence
$$0\to (M')^{\vee}\to (M^{\vee\vee})^{\vee}
\to M^{\vee} \to \Ext^1_A(M',A)\to \cdots .$$
Since $j(M')\geq 2$, we have $(M')^{\vee}
=\Ext^1_A(M',A)=0$. Therefore $(M^{\vee\vee})^{\vee}$
is naturally isomorphic to $M^{\vee}$ as required.

(2) By Lemma \ref{xxlem2.7}, the Gabber closure of $M$ is
isomorphic to $M^{\vee\vee}$. The assertion follows from
part (1).
\end{proof}

\section{A NQR of an algebra}
\label{xxsec3}
In this section, we introduce the notion of a noncommutative
quasi-resolution (NQR), which is a further generalization of
the notion of a NCCR, and then study some basic properties.

Let ${\mathcal A}$ be a category consisting of a class of noetherian
$\Bbbk$-algebras such that $A$ is in ${\mathcal A}$ if and
only if $A^{\op}$ is in ${\mathcal A}$. Together with
${\mathcal A}$ we consider a special class of
modules/morphisms/bimodules. Our definition of a noncommutative
quasi-resolution will be made inside the category ${\mathcal A}$.
Sometimes it is necessary to be specific, but in the most of
cases, it is quite easy to understand what is the setting of
${\mathcal A}$. We also need to specify or fix a dimension
function $\partial$. Here are a few examples.

\begin{example}
\label{xxex3.1}
\begin{enumerate}
\item[(1)]
Let ${\mathcal A}$ be the category of ${\mathbb N}$-graded locally
finite noetherian $\Bbbk$-algebras with finite GK-dimension. We
only consider graded modules. An $(A,B)$-bimodule is a
${\mathbb Z}$-graded module that has both left graded $A$-module
and right graded $B$-module structures. The dimension function
$\partial$ is chosen to be $\GK$.
\item[(2)]
We might modify the category in part (1) by restricting algebras
to those with balanced Auslander dualizing complexes in the sense
of \cite{YZ2}. In this case, we might take the dimension function
to be a constant shift of the canonical dimension defined in
\cite[Definition 2.9]{YZ2}.
\item[(3)]
Let $R$ be a noetherian commutative algebra with finite Krull
dimension. Let ${\mathcal A}$ be
the category of algebras that are module-finite $R$-algebras.
Modules are usual modules, but an $(A,B)$-bimodule means an
$R$-central $(A,B)$-bimodule. The dimension function in this case
could be the Krull dimension.
\end{enumerate}
\end{example}

Unless otherwise stated, we retain Hypothesis \ref{xxhyp1.3}
concerning the fixed dimension function $\partial$ for modules
over $A$, $B$ and $B_i$, a bimodule (such as $M$ or $N$ in most
of cases) over these rings in this section (including Definitions
\ref{xxdef3.2} and \ref{xxdef3.16}) is finitely generated on both sides.
As a consequence of these assumptions, $\partial(M)$ can be defined
by considering $M$ as either a left or a right module. Therefore
$M\cong_{n}N$ is well-defined on either left or right sides for
another bimodule $N$. If we use other rings such as $D$, we may not
assume these hypotheses.

Here is our main definition.

\begin{definition}
\label{xxdef3.2}
Let $A\in {\mathcal A}$ be an algebra with $\partial(A)=d$. Let
$s$ be an integer between $0$ and $d-2$.
\begin{enumerate}
\item[(1)]
If there are an Auslander regular $\partial$-$\CM$ algebra
$B\in {\mathcal A}$ with $\partial(B)=d$ and two bimodules
$_B\!M_A$ and $_A\!N_B$ such that
$$M\otimes_A N\cong_{d-2-s}B,\quad N\otimes_B M\cong_{d-2-s}A$$
as bimodules, then the triple $(B,_B\!M_A,_A\!N_B)$ is
called an {\it $s$-noncommutative quasi-resolution}
({\it $s$-NQR} for short) of $A$.
\item[(2)](Definition 0.5)
A $0$-noncommutative quasi-resolution ($0$-NQR) of $A$
is called a {\it noncommutative quasi-resolution}
({\it NQR} for short) of $A$.
\end{enumerate}
\end{definition}

We will see that the notion of a NQR is a generalization
of the notion of a NCCR in Section 7. First we prove
the following lemmas.

\begin{lemma}
\label{xxlem3.3}
Let $B$ be a $\partial$-$\CM$ algebra with $\partial(B)=d$.
\begin{enumerate}
\item[$(1)$]
If there exist $B$-modules $M$ and $N$ such that
$M\cong_{d-2}N$, then $M^{\vee}\cong N^{\vee}$, where
$(-)^{\vee}:=\Hom_B(-,B).$
\item[$(2)$]
Let $D$ be another algebra and supposed that $M$ and $N$
are $(D,B)$-bimodules such that $M\cong_{d-2}N$ as
$(D,B)$-bimodules. Then $M^{\vee}\cong N^{\vee}$ as
$(B,D)$-bimodules.
\end{enumerate}
\end{lemma}

\begin{proof} The proofs of parts (1) and (2) are similar.
We only prove part (1).

By definition, there exists a right $B$-module $P$ and
$B$-module morphisms $f: M\to P$ and $g: N\to P$ such that
both the kernel and cokernel of $f$ and $g$ have $\partial$-dimension
no more than $d-2$. It suffices to show that $P^{\vee}
\cong M^{\vee}$. Without loss of generality, we assume that
$f: M\to N$ is a right $B$-morphism such that the kernel and
cokernel of $f$ have $\partial$-dimension no more than $d-2$.
By the properties of $f$ we have two exact
sequences
$$0\to Q\to N\to C\to 0$$
and
$$0\to K \to M\to Q\to 0,$$
where $Q=\im f$ and where $C$ and $K$ have $\partial$-dimension
no more than $d-2$. We need to show that $Q^{\vee}\cong N^{\vee}$
and $M^{\vee}\cong Q^{\vee}$. The proofs of these assertions are
similar, we only show the first one. Applying $\Hom_B(-,B)$ to the
first short exact sequence, we obtain that a long exact sequence
$$0\rightarrow \Hom_B(C,B)\rightarrow N^{\vee}\rightarrow Q^{\vee}
\rightarrow \Ext_B^1(C,B)\rightarrow\cdots.$$
Since $\partial(C)+j(C)=\partial(B)=d$ and $\partial(C)\leq d-2$,
$j(C)\geq 2$, which means that $\Hom_B(C,B)=0=\Ext_B^1(C,B)$.
So, $N^{\vee}\cong Q^{\vee}$. Similarly, $P^{\vee}\cong Q^{\vee}$.
Therefore, $P^{\vee}\cong N^{\vee}$.
\end{proof}

The following corollary is clear.

\begin{corollary}
\label{xxcor3.4}
Let $B$ be a $\partial$-$\CM$ algebra with $\partial(B)=d$.
\begin{enumerate}
\item[$(1)$]
If $M,N\in\refl B$ such that $M\cong_{d-2}N$, then
$M\cong N$.
\item[$(2)$]
Let $D$ be another algebra and supposed that $M$
and $N$ are $(D,B)$-bimodules such that $M\cong_{d-2}N$
as $(D,B)$-bimodules. If $M,N\in\refl B$, then
$M\cong N$ as $(B,D)$-bimodules.
\end{enumerate}
\end{corollary}

\begin{proof} We prove (2). Since $M\cong_{d-2}N$,
Lemma \ref{xxlem3.3} implies that $M^{\vee}\cong N^{\vee}$
as $(B,D)$-bimodules. The assertion follows by applying
$(-)^{\vee}$ and the fact that $M,N\in\refl B$.
\end{proof}

\begin{lemma}
\label{xxlem3.5}
Let $A$ and $B$ be algebras and $n\in\N$. Suppose that there
exist two bimodules $_B\!M_A$ and $_A\!N_B$ such that
$$M\otimes_AN\cong_{n}B,\quad N\otimes_BM\cong_{n}A$$
as bimodules. If $\partial$ satisfies $\gamma_{n,1}(M)^l$ and
$\gamma_{n,1}(N)^l$,  then
$$\qmod_nA\cong\qmod_nB.$$
\end{lemma}

\begin{proof}
Let $\pi$ be the natural functor
$\modu A\longrightarrow\qmod_n A$
(or $\modu B\longrightarrow\qmod_n B$).
Denote by $\mathcal{M}:=\pi(M)$ and $\mathcal{N}:=\pi(N)$.
By Lemma \ref{xxlem1.9}, we have two well-defined functors
$$F(-):=-\otimes_{\mathcal{A}}\mathcal{N}:\qmod_nA\longrightarrow\qmod_nB$$
and
$$G(-):=-\otimes_{\mathcal{B}}\mathcal{M}:\qmod_nB\longrightarrow\qmod_nA.$$
By Lemma \ref{xxlem1.9} again,
$$F\circ G(-)=-\otimes_{\mathcal{B}}\mathcal{M}\otimes_{\mathcal{A}}
\mathcal{N}\cong-\otimes_{\mathcal{B}}\pi(M\otimes_A N)
\cong -\otimes_{\mathcal{B}}\mathcal{B},$$
and
$$ G\circ F(-)=-\otimes_{\mathcal{A}}\mathcal{N}\otimes_{\mathcal{B}}
\mathcal{M}\cong-\otimes_{\mathcal{A}}\pi(N\otimes_B M)
\cong -\otimes_{\mathcal{A}}\mathcal{A}.$$
Therefore $F$ and $G$ are equivalences, in other words,
$\qmod_nA\cong\qmod_nB$.
\end{proof}

The equivalence $\qmod_nA\cong\qmod_nB$ in the above lemma 
can be considered as a noncommutative Fourier-Mukai transform 
between two noncommutative spaces. We refer to \cite{Hu} for 
the classical setting.

\begin{remark}
\label{xxrem3.6}
The above lemma holds true for a $NQR$ $(B,_B\!M_A,_A\!N_B)$ of
an algebra $A$ when $\partial$ satisfies $\gamma_{d-2,1}(M)^l$
and $\gamma_{d-2,1}(N)^l$.
\end{remark}

For an $n$-pure $(A,B)$-bimodule $M$, the Gabber closure
of $_AM$ is denoted by $G_{A^{\op}}(M)$ and the
Gabber closure of $M_B$ is denoted by $G_B(M)$. We consider
both $G_{A^{\op}}(M)$ and $G_B(M)$ as extensions of $M$.

\begin{lemma}
\label{xxlem3.7}
Let $A$ and $B$ be Auslander-Gorenstein and $\partial$-CM
algebras with
$$\partial(A)=\partial(B)=d.$$
Assume Hypothesis
\ref{xxhyp1.3} holds. Let $n$ be an integer. Let $M$ denote
an $(A,B)$-bimodule that is finitely generated on both sides.
\begin{enumerate}
\item[$(1)$]
Let $M$ be $n$-pure on both sides.
Then $G_{A^{\op}}(M)\cong G_B(M)$ naturally as bimodules with
restriction on $M$ being the identity.
\item[$(2)$]
Let $M$ be $n$-pure on both sides.
Then $G_B(M)=M$ if and only if $G_{A^{\op}}(M)=M$.
\item[$(3)$]
The $(A,B)$-bimodule $M$ is reflexive on the
left if and only if it is reflexive on the right.
\end{enumerate}
\end{lemma}

\begin{proof} Without loss of generality, we can assume
that $\partial$ is the canonical dimension.

(1) By Proposition \ref{xxpro2.10}(6), $G_B(M)$ is an
$(A,B)$-bimodule and $g_{M_B}: M\to G_B(M)$ is a bimodule
morphism. By Proposition \ref{xxpro2.10}(3), $G_B(M)$ is
finitely generated on the right. We claim that
\begin{enumerate}
\item[(a)]
$g_{M_B}$ is an essential extension of $M$ on the left,
\item[(b)]
$G_B(M)$ is finitely generated over the left, and
\item[(c)]
$g_{M_B}$ is a tame and pure extension of $M$ on the left.
\end{enumerate}

By Hypothesis \ref{xxhyp1.3}(3),
\begin{equation}
\label{E2.7.1}\tag{E2.7.1}
\partial(_A(G_{B}(M)/M))\leq \partial((G_{B}(M)/M)_B)\leq
d-n-2.
\end{equation}
To prove (a) let $S$ be a left $A$-submodule of $G_{B}(M)$
such that $S\cap M=0$. Then $S$ is isomorphic to a submodule
of $G_B(M)/M$. As a consequence, $\partial(S)\leq d-n-2$.
Let $U$ be the largest left $A$-submodule of $G_M(B)$ with
$\partial\leq d-n-2$. Then $U\cap M=0$ and $U$ is also a right
$B$-submodule. If $U\neq 0$, it contradicts the fact that
$G_B(M)$ is an essential extension of $M$ on the right.
Therefore $U=0$ and $S=0$. Thus Claim (a) is proven.

Claim (b) follows from Lemma \ref{xxlem2.9} and \eqref{E2.7.1}.

Claim (c) follows from Claim (b) and and \eqref{E2.7.1}.

Next we consider the Gabber closure of the module
$N:=G_{B}(M)$ on the left. By Proposition \ref{xxpro2.10}(6),
$G_{A^{\op}}(N)$ is an $(A,B)$-bimodule and $g_{_{A^{\op}}N}: N\to
G_{A^{\op}}(N)$ is a bimodule morphism. By symmetric,
$g_{_{A^{\op}}N}$ has properties (a,b,c) on the right. By part (c),
$g_{_{A^{\op}}N}$ is a tame and pure extension of $N$ on the
right. By Proposition \ref{xxpro2.10}(4), $N_B(:=G_B(M))$ does
not have a proper tame and pure extension on the right.
Therefore $G_{A^{\op}}(N)=N$. This implies that $_AN$ does not
have a proper tame and pure extension. Thus $N$ must be
$G_{A^{\op}}(M)$ by Proposition \ref{xxpro2.10}(5).

(2) This is a consequence of part (1).

(3) By Lemma \ref{xxlem2.7}, $M_B$ is reflexive if and only
if $M$ is 0-pure and $G_B(M)=M$. The assertion follows by part (2).
\end{proof}

\begin{hypothesis}
\label{xxhyp3.8}
We are continuing to work with algebras in a given
category ${\mathcal A}$ with a fixed dimension
function $\partial$ defined for all modules over
rings in ${\mathcal A}$. As indicated at the beginning of this
section we assume Hypothesis \ref{xxhyp1.3} for all algebras
in ${\mathcal A}$. Now we further assume that $\partial$
satisfies $\gamma_{d-2,1}(A,B)$ for algebras $A$ and $B$
in ${\mathcal A}$ with $d=\partial(A)=\partial(B)$, which
covers the hypotheses in Lemmas \ref{xxlem1.9} and \ref{xxlem1.10}.
\end{hypothesis}

\begin{proposition}
\label{xxpro3.9}
Assume Hypothesis \ref{xxhyp3.8}.
Let $B_i$ be Auslander-Gorenstein and $\partial$-$\CM$ algebras
with $\partial(B_i)=d$ for $i=1,2$. Suppose that there are
bimodules $_{B_1}\!T_{B_2}$ and $_{B_2}\!\widetilde{T}_{B_1}$
{\rm{(}}finitely generated on both sides{\rm{)}} such that
\begin{equation}
\label{E3.9.1}\tag{E3.9.1}
T\otimes_{B_2}\widetilde{T}\cong_{d-2}B_1
\end{equation}
and
\begin{equation}
\notag
\widetilde{T}\otimes_{B_1}T\cong_{d-2}B_2.
\end{equation}
Then there exist  $_{B_1}\!U_{B_2}$ and  $_{B_2}\!V_{B_1}$
{\rm{(}}finitely generated on both sides{\rm{)}}
such that $U,V$ are reflexive modules on both sides and
$$U\otimes_{B_2}V\cong_{d-2}B_1,\quad V\otimes_{B_1}U\cong_{d-2}B_2.$$
In other words, we can replace $T$ and $\widetilde{T}$
with $_{B_1}\!U_{B_2}$ and  $_{B_2}\!V_{B_1}$ respectively,
which are reflexive modules on both sides.
\end{proposition}

In the following proof, we need to deal with multiple different
rings/modules. It is convenient to fix the following
notation specially when we deal with bimodules.
Starting from a right $B$-module $M$,
we use $M^{\vee}$ (respectively, $M^{\vee\vee}$) for
$\Hom_{B}(M,B)$ (respectively, $\Hom_{B^{\op}}(\Hom_B(M,B),B)$).
Starting from a left $B$-module $M$,
we use ${^{\vee} M}$ (respectively, ${^{\vee\vee} M}$) for
$\Hom_{B^{\op}}(M,B)$ (respectively, $\Hom_{B}(\Hom_{B^{\op}}(M,B),B)$).
For example, for a $(B_1,B_2)$-bimodule $M$, we have
$$M^{\vee\vee}=\Hom_{B_2^{\op}}(\Hom_{B_2}(M,B_2),B_2)$$
and
$${^{\vee\vee} M}=\Hom_{B_1}(\Hom_{B_1^{\op}}(M,B_1),B_1).$$
By Lemma \ref{xxlem2.19}, for every finitely
generated right $B$-module $M$, $M^{\vee\vee}$ is
(either zero or) always reflexive when $B$ is
Auslander-Gorenstein.

\begin{proof}[Proof of Proposition \ref{xxpro3.9}]
By Lemma \ref{xxlem3.3}(2) and \eqref{E3.9.1}, we have
$$^{\vee\vee}\!(T\otimes_{B_2}\widetilde{T})
\cong ^{\vee\vee}\!B_1\cong B_1,$$
as $B_1$-bimodules. Hence there is a composite map
$$\psi: T\otimes_{B_2}\widetilde{T}
\longrightarrow ^{\vee\vee}\!(T\otimes_{B_2}\widetilde{T})
\longrightarrow B_1$$
which induces the $(d-2)$-isomorphism from
$T\otimes_{B_2}\widetilde{T}$ to $B_1$.

Define
$$\tau(T):=\{x\in T|xr=0\, \,
\text{for \,some \,regular \,element\, }\, r\in B_2 \}.$$
By \cite[Proposition 2.4(4) and Theorem 6.1]{ASZ1},
$\tau(T)$ is the maximal torsion $B_2$-submodule of $T$ such that
$\partial(\tau(T))$ is at most $d-1$, namely,
$\tau(T)\in\modu_{d-1}B_2$. Since we assume that $\partial$ is
symmetric (Hypothesis \ref{xxhyp1.3}(3)), $\tau(T)\in
\modu_{d-1} B_1^{\op}$.

Note that $\widetilde{T}$ is a finitely generated left $B_2$-module,
and by the definition of $\tau(T)$, we have
$\tau(T)\otimes_{B_2}\widetilde{T}\in\modu_{d-1}{B_1^{\op}}$.
Applying $-\otimes_{B_2}\widetilde{T}$ to an exact sequence
$$0\rightarrow\tau(T)\rightarrow T\rightarrow T/\tau(T)\rightarrow0$$
in $\modu B_2,$ one has an exact sequence
$$\tau(T)\otimes_{B_2}\widetilde{T}\xrightarrow{f} T\otimes_{B_2}
\widetilde{T}\rightarrow T/\tau(T)\otimes_{B_2}\widetilde{T}
\rightarrow0$$ in $\modu B_1$. Then
\begin{equation}
\notag
T/\tau(T)\otimes_{B_2}\widetilde{T}\cong (T\otimes_{B_2}\widetilde{T})/\im(f)
\end{equation}
in $\modu B_1$.  Since $B_1\in\refl B_1$ is a $0$-pure module and
$\im(f)\subseteq T\otimes_{B_2}\widetilde{T}$, we have $\psi(\im(f))=0$,
whence there are well-defined morphisms
$$T\otimes_{B_2}\widetilde{T} \to T/\tau(T)\otimes_{B_2}\widetilde{T}\cong
(T\otimes_{B_2}\widetilde{T})/\im(f)\longrightarrow
B_1\cong_{d-2} T\otimes_{B_2}\widetilde{T}$$
such that the composition is a $(d-2)$-isomorphism.
Therefore, $T/\tau(T)\otimes_{B_2}\widetilde{T}\cong_{d-2}B_1$.
Now, we can replace $T$ with $T/\tau(T)$ in \eqref{E3.9.1} and
assume that $\tau(T)=0$, namely, $T$ is a 0-pure $B_2$-module
(whence a 0-pure left $B_1$-module by the symmetry of $\partial$).

Let $U:=G_{B_2}(T)$ be the Gabber closure of $T$. By Lemma
\ref{xxlem3.7}, $U$ is isomorphic to $G_{B_1^{\op}}(T)$
as bimodules. This implies that $U$ is finitely generated
on both sides and $T\cong_{d-2} U$ by the definition of
the Gabber closure. Combining this $(d-2)$-isomorphism with
\eqref{E3.9.1} and Lemma \ref{xxlem1.10}, we have
$$U\otimes_{B_2}\widetilde{T}\cong_{d-2} B_1.$$
Similarly,
$$\widetilde{T}\otimes_{B_1}V\cong_{d-2}B_2.$$
Since $T$ is 0-pure, by Lemmas \ref{xxlem2.19} and
\ref{xxlem3.7}(3), $U$ is reflexive on both sides.
Next we take $V=G_{B_1}({\widetilde{T}})$ and repeat
the above argument. It is easy to see that $U$ and $V$
satisfy the required conditions.
\end{proof}

\begin{lemma}
\label{xxlem3.10}
Let $B$ be an Auslander-Gorenstein and $\partial$-$\CM$ algebra
with $\partial(B)=d$ and $U$ a nonzero reflexive $B$-module. Then
\begin{enumerate}
\item[$(1)$]
$\Hom_B(C,U)=0$ for any $C\in\modu_{d-1}B.$
\item[$(2)$]
$\Ext_B^1(K,U)=0$ for any $K\in\modu_{d-2}B.$
\end{enumerate}
\end{lemma}

\begin{proof}
(1) Let $f\in\Hom_B(C,U)$.  Since $\im(f)$ is a quotient module of $C$,
$$\partial(\im(f))\leq\partial(C)\leq d-1.$$
By Corollary \ref{xxcor2.16}, $U$ is $0$-pure. If $\im(f)\neq 0$,
then $\partial(\im(f))=\partial(U)=d$, a contradiction.
Therefore $\im(f)=0$, which implies that $\Hom_B(C,U)=0$.

(2) If $\Ext_B^1(K,U)\neq 0$, there is a non-split extension
\begin{equation}
\label{E3.10.1}\tag{E3.10.1}
0\rightarrow U\rightarrow E\rightarrow K\rightarrow0
\end{equation}
in $\modu B$. Let $\tau(E)$ be the maximal submodule of $E$ such
that $\partial(\tau(E))\leq d-1$. Then there exists an induced
morphism $\varphi:\tau(E)\rightarrow K$ such that
$\ker\varphi\subseteq U\cap\tau(E).$ Since
$\partial(U\cap\tau(E))\leq\partial(\tau(E))\leq d-1$ and $U$ is $0$-pure,
$U\cap\tau(E)=0$. This implies that $\varphi$ is injective, whence, we
can consider $\tau(E)$ as a submodule of $K$, and $\varphi$ is not
surjective (following by the fact that \eqref{E3.10.1} is non-split).
Hence, we obtain a short exact sequence
$$0\longrightarrow U\longrightarrow E/\tau(E)
\longrightarrow K/\tau(E)\longrightarrow0$$
with
$$\partial\Big(E/\tau(E)\Big/U\Big)=\partial(K/\tau(E))
\leq\partial(K)\leq d-2=\partial(U)-2.$$
By the definition of $\tau(E)$, $E/\tau(E)$ is $0$-pure. So,
$E/\tau(E)$ is a tame and pure extension of $U$.
By hypothesis, $U$ is a reflexive module, whence $U=G_B(U)$
by Lemma \ref{xxlem2.7}. Then, by Proposition \ref{xxpro2.10}(4),
$E/\tau(E)\cong U$, or equivalently, $K/\tau(E)=0$.
This means that $K=\tau(E)$, or equivalently, the exact sequence
\eqref{E3.10.1} is split, a contradiction. The assertion follows.
\end{proof}

\begin{remark}
\label{xxrem3.11}
The reflexivity of module $U$ is not necessary for Lemma
\ref{xxlem3.10}(1). In fact, when $U$ is a $0$-pure module,
Lemma \ref{xxlem3.10}(1) is also true.
\end{remark}

\begin{lemma}
\label{xxlem3.12}
Let $B$ be an Auslander-Gorenstein and $\partial$-$\CM$
algebra with $\partial(B)=d$.
Suppose that $0\neq U\in\modu B$ satisfies
$$\Hom_B(N,U)=0=\Ext_B^1(N,U)$$
for all $N\in\modu_{d-2}B.$
\begin{enumerate}
\item[$(1)$]
For $M\in \modu B$,
$\Hom_{\qmod_{d-2}B}({\mathcal M},{\mathcal U})\cong
\Hom_{B}(M,U)$.
\item[$(2)$]
$\End_{\qmod_{d-2}B}(\mathcal{U})\cong\End_B(U)$.
\item[$(3)$]
In particular, if $M\in \modu B$ and $U\in\refl B$, then
$$\Hom_{\qmod_{d-2}B}({\mathcal M},{\mathcal U})=
\Hom_{B}(M,U)$$
and
$$\End_{\qmod_{d-2}B}(\mathcal{U})\cong\End_B(U).$$
\end{enumerate}
\end{lemma}

\begin{proof}
(1) By the assumption, $U$ does not have any nonzero $B$-submodule
of $\partial$-dimension at most $d-2$. Combining with
\eqref{E1.4.2}, we have
$$\Hom_{\qmod_{d-2}B}({\mathcal M},{\mathcal U})=
\lim_{\longrightarrow}\Hom_B(K,U),$$
where the limit runs over all the submodules $K\subseteq M$
such that $\partial(M/K)\leq d-2$.
The functor $\pi$ induces a natural morphism
\begin{equation}
\label{E3.12.1}\tag{E3.12.1}
\phi_{\pi}: \quad \Hom_{B}(M,U)\to
\Hom_{\qmod_{d-2}B}({\mathcal M},{\mathcal U}):=
\lim_{\longrightarrow}\Hom_B(K,U),
\end{equation}
where $K\subseteq M$ as described as above. By hypotheses,
$$\Hom_B(M/K,U)=0=\Ext_B^1(M/K,U).$$
Now the short exact sequence
$0\rightarrow K\rightarrow M\rightarrow M/K\rightarrow0$
induces a long exact sequence
$$0\rightarrow\Hom_B(M/K,U)\rightarrow\Hom_B(M,U)\rightarrow
\Hom_B(K,U)\rightarrow\Ext_B^1(M/K,U)\rightarrow\cdots,$$
which implies that $\Hom_B(M,U)\cong\Hom_B(K,U)$ for all $K$.
Thus $\phi_{\pi}$ in \eqref{E3.12.1} is an isomorphism.
The assertion follows.

(2) Take $M=U$ in \eqref{E3.12.1}, the functor
$\pi$ induces a morphism of algebras $\phi_{\pi}$. By part (1),
$\phi_{\pi}$ is also an isomorphism of $\Bbbk$-vector spaces.
The assertion follows.

(3) If $U\in\refl B,$ then by Lemma \ref{xxlem3.10},
$$\Hom_B(N,U)=0=\Ext_B^1(N,U).$$
The assertion follows from parts (1,2).
\end{proof}

\begin{lemma}
\label{xxlem3.13}
Assume Hypothesis \ref{xxhyp3.8}. Let $_{B_1}\!U_{B_2}$
be the module appeared in Proposition \ref{xxpro3.9}. Then
it is a reflexive module on both sides such that
$B_1\cong\End_{B_2}(U)$ and
$B_2^{\op}\cong\End_{B_1^{\op}}(U)$.
\end{lemma}

\begin{proof}
By Proposition \ref{xxpro3.9} and Lemma \ref{xxlem3.5},
$_{B_1}\!U_{B_2}$ is a reflexive module on both sides
and induces the following equivalence of categories
$$F:= -\otimes_{\mathcal{B}_1}\mathcal{U}:
\qmod_{d-2}B_1\longrightarrow \qmod_{d-2}B_2.$$
Since $F$ is an equivalence functor, we obtain isomorphisms of algebras:
\begin{eqnarray*}
\End_{\qmod_{d-2}B_1}(\mathcal{B}_1)\cong\End_{\qmod_{d-2}B_2}
(F(\mathcal{B}_1))=\End_{\qmod_{d-2}B_2}(\mathcal{U}).
\end{eqnarray*}
Now it suffices to show that
$$\End_{\qmod_{d-2}B_1}(\mathcal{B}_1)\cong B_1$$
and
$$\End_{\qmod_{d-2}B_2}(\mathcal{U})\cong\End_{B_2}(U).$$
Since $B_1\in\refl B_1$ and $U\in\refl B_2$, by Lemma
\ref{xxlem3.12}(3), the above isomorphisms hold, as required.

By symmetry, $B_2^{\op}\cong\End_{B_1^{\op}}(U)$.
\end{proof}

\begin{corollary}
\label{xxcor3.14}
Assume Hypothesis \ref{xxhyp3.8}.
Let $A$ be an Auslander-Gorenstein and $\partial$-$\CM$ algebra
with $\partial(A)=d$.
Let $(B,_B\!M_A,_A\!N_B)$ be a $NQR$ of $A$. Then
there exists a bimodule $_{B}\!U_{A}:=M^{\vee\vee}$
which is a reflexive module on both sides such that
$$B\cong\End_{A}(U) \quad {\text{and}}\quad
A^{\op}\cong \End_{B^{\op}}(U).$$
\end{corollary}

\begin{theorem}
\label{xxthm3.15}
Assume Hypothesis \ref{xxhyp3.8}.
Let $A$ be an algebra with $\partial(A)=d$. Suppose that $A$ has two
NQRs $(B_i, _{B_i}\!(M_i)_{A}, _{A}\!(N_i)_{B_i})$ for $i=1,2$.
Then there exists a bimodule
$_{B_1}\!U_{B_2}$ which is a reflexive module on both sides such that
$$B_1\cong\End_{B_2}(U) \quad {\text{and}} \quad
B_2^{\op}\cong\End_{B_1^{\op}}(U).$$
\end{theorem}

\begin{proof}
Let $T:=M_1\otimes_A N_2$ and $\widetilde{T}:=M_2\otimes_A N_1$.
Then there are isomorphisms, by Lemma \ref{xxlem1.10},

\begin{eqnarray*}
T\otimes_{B_2}\widetilde{T}\cong_{d-2}B_1,
\end{eqnarray*}
and
\begin{eqnarray*}
\widetilde{T}\otimes_{B_1}T\cong_{d-2}B_2.
\end{eqnarray*}
Thus, the result follows from  Lemma \ref{xxlem3.13}.
\end{proof}

Finally we introduce another definition, which is
a bit closer to Van den Bergh's NCCR.

\begin{definition}
\label{xxdef3.16}
Let $A\in {\mathcal A}$ be an Auslander-Gorenstein algebra with
$\partial(A)=d$. Let $s$ be an integer between $0$ and $d-2$.
\begin{enumerate}
\item[(1)]
If there are an Auslander regular $\partial$-$\CM$ algebra
$B\in {\mathcal A}$
with $\partial(B)=d$ and two bimodules $_B\!M_A$ and
$_A\!N_B$ which are reflexive on both sides such that
$$M\otimes_A N\cong_{d-2-s}B,\quad N\otimes_B M\cong_{d-2-s}A$$
as bimodules, then the triple $(B,_B\!M_A,_A\!N_B)$ is
called an {\it $s$-noncommutative quasi-crepant resolution}
({\it $s$-NQCR} for short) of $A$.
\item[(2)]
A $0$-noncommutative quasi-crepant resolution ($0$-NQCR) of $A$
is called a {\it noncommutative quasi-crepant resolution}
({\it NQCR} for short) of $A$.
\end{enumerate}
\end{definition}

By definition, a NQCR of an Auslander-Gorenstein algebra $A$ is
automatic a NQR of $A$. Suppose $A$ is an Auslander-Gorenstein
and $\partial$-$\CM$ algebra. If $A$ has a NQR, then, by
Proposition \ref{xxpro3.9}, $A$ has a NQCR. However, it is not
clear to us whether an $s$-NQR (Definition \ref{xxdef3.2})
produces an $s$-NQCR when $s>0$.

\section{NQRs in dimension two}
\label{xxsec4}

With the preparation in the last few sections, we are ready to prove
a version of part (1) of the main theorem.

\begin{lemma}
\label{xxlem4.1}
Let $A$ be an Auslander-Gorenstein and $\partial$-$\CM$ algebra.
Then
$$\injdim A\leq \partial(A).$$
\end{lemma}

\begin{proof}
Let $d=\injdim A$. Then there is a right $A$-module $M$ such that
$_A N:=\Ext^d_A(M,A)\neq 0$. By the Auslander condition, $j(N)\geq d$.
Now, by the $\partial$-$\CM$ property,
$$\partial (A) =\partial(N)+j(N)\geq d=\injdim A.$$
\end{proof}

\begin{theorem}
\label{xxthm4.2}
Assume Hypothesis \ref{xxhyp3.8}.
Suppose that $(B_i, _{B_i}\!(M_i)_{A}, _{A}\!(N_i)_{B_i})$ are two $NQRs$
of $A$ for $i=1,2$. If $\partial(A)\leq 2$, then $B_1$
and $B_2$ are Morita equivalent.
\end{theorem}

\begin{proof}
By definition, $\partial(B_i)=\partial(A)\leq 2$. By
Lemma \ref{xxlem4.1},
$$\gldim (B_i)=\injdim (B_i)\leq
\partial (B_1)\leq 2.$$

Let $T:=M_1\otimes_A N_2$ and $\widetilde{T}:=M_2\otimes_A N_1$.
Then there are isomorphisms, by Lemma \ref{xxlem1.10},
\begin{eqnarray*}
T\otimes_{B_2}\widetilde{T}\cong_{d-2}B_1,
\end{eqnarray*}
and
\begin{eqnarray*}
\widetilde{T}\otimes_{B_1}T\cong_{d-2}B_2.
\end{eqnarray*}
By Proposition \ref{xxpro3.9}, there exist $_{B_1}\!U_{B_2}$
and $_{B_2}\!V_{B_1}$ which are reflexive modules (and
finitely generated) on both sides such that
$$U\otimes_{B_2}V\cong_{d-2} B_1\quad {\text{and}}\quad
V\otimes_{B_1}U\cong_{d-2} B_2.$$
Since $\gldim (B_i)\leq 2$, by Corollary \ref{xxcor2.13},
$U$ and $V$ are projective modules on both
sides. Hence $U\otimes_{B_2}V$ and $V\otimes_{B_1}U$ are projective (whence
reflexive) on both sides. Therefore, by Corollary \ref{xxcor3.4}, we have
$$U\otimes_{B_2}V\cong B_1\quad {\text{and}}\quad
V\otimes_{B_1}U\cong B_2,$$
which implies that $B_1$ and $B_2$ are Morita equivalent.
This finishes the proof.
\end{proof}

\section{Depth in the noncommutative setting}
\label{xxsec5}
The proof of part (2) of the main theorem needs some extra preparation.
In particular, it uses the concept of a depth in noncommutative algebra.
There are several slightly different definitions of the depth in the
noncommutative setting. It is a good idea to fix some notation.

Let $A$ be an algebra with a dimension function $\partial$.

\begin{hypothesis}
\label{xxhyp5.1}
Let $A$ be an algebra. Assume that $\modu_0A\neq 0$, namely,
there is a nonzero module $S\in\modu_0 A$.
\end{hypothesis}

Hypothesis \ref{xxhyp5.1} is sometimes quite natural, but
not automatic. By abuse of notation, we can also talk about
Hypothesis \ref{xxhyp5.1} for a single algebra $A$ or for a
family of algebras ${\mathcal A}$.

\begin{definition}
\label{xxdef5.2}
Let $A$ be an algebra and $\partial$ be a dimension function. For
an $A$-module $M\in\modu A,$ define
\begin{eqnarray*}
\dep_AM = \inf\{i|\Ext_A^i(S,M)\neq0 \,\text{\,for\,some}\,\,
S\in\modu_0A\} \in {\mathbb N}\cup\{+\infty\}.
\end{eqnarray*}
If no confusion can arise, we write $\dep M$ for $\dep_AM$. If
Hypothesis \ref{xxhyp5.1} fails for $A$, then $\dep_A M=+\infty$
for every $A$-module $M$.
\end{definition}

If $\dep_A M<+\infty$ for some $A$-module $M$, then Hypothesis
\ref{xxhyp5.1} holds for the algebra $A$. One can easily prove
the following depth lemma.

\begin{lemma}
\label{xxlem5.3}
Let $A$ be an algebra and $\partial$ be a dimension function.
Let
$$0\rightarrow M'\rightarrow M\rightarrow M''\rightarrow0$$
be a short exact sequence of finitely generated right
$A$-modules. Then
\begin{enumerate}
\item[$(1)$]
$\dep M\geq\min\{\dep M',\dep M''\}$.
\item[$(2)$]
$\dep M'\geq\min\{\dep M,\dep M''+1\}$.
\item[$(3)$]
$\dep M''\geq\min\{\dep M,\dep M'-1\}$.
\end{enumerate}
\end{lemma}

The proof of Lemma \ref{xxlem5.3}(2) is basically
given in the proof of Lemma \ref{xxlem5.5}.

The following proposition resembles the 
``special $\chi$ condition'' in \cite[Definition 16.5.16]{Y2}.

\begin{proposition}
\label{xxpro5.4}
Suppose that Hypothesis \ref{xxhyp5.1} holds
for $A$. If $A$ is a $\partial$-$\CM$ algebra with
$\partial(A)=\partial(A^{\op})=d$, then
$$\dep_AA=\dep_{A^{\op}}A=d.$$
\end{proposition}

\begin{proof}
Given every nonzero $S\in\modu_0A$, we have
$$j_A(S)=\partial(A)-\partial(S)=\partial(A)=d,$$
namely,
$$d=\inf\{i|\Ext_A^i(S,A)\neq0\}.$$
Thus $\dep_A A=d$. Similarly, we have $\dep_{A^{\op}}A=d$.
\end{proof}

The proof of Theorem \ref{xxthm0.6}(2) also uses the following
two lemmas, which were known in the local or graded
setting \cite[Lemma 3.15]{CKWZ2}.

\begin{lemma}
\label{xxlem5.5}
Suppose that $M$ and $N$ are nonzero
finitely generated $A$-modules related by the exact sequence
$$0\longrightarrow M\longrightarrow P_{s-1}\longrightarrow
P_{s-2}\longrightarrow\cdots\longrightarrow P_0
\longrightarrow N\longrightarrow0.$$
Then $$\dep_A(M)\geq\min\{\dep_A(N)+s,\dep_A(P_0),
\ldots,\dep_A(P_{s-2}),\dep_A(P_{s-1})\}.$$
If, further, $\dep_A(P_j)\geq s+\dep_A(N)$ for each $j$,
then $\dep_A(M)=\dep_A(N)+s.$
\end{lemma}

\begin{proof} There is nothing to be proved if Hypothesis
\ref{xxhyp5.1} fails for $A$. So we assume that Hypothesis
\ref{xxhyp5.1} holds for $A$ for the rest of the proof. By
induction on $s$, it suffices to show the assertion in the
case of $s=1$. For any $S\in\modu_0 A$, letting $P=P_0$ and
applying $\Hom_A(S,-)$ to the short exact sequence
$$0\rightarrow M\rightarrow P\rightarrow N\rightarrow0,$$
we obtain a long exact sequence
$$\begin{aligned}
\cdots\rightarrow\Ext_A^{i-1}(S,P)&\rightarrow\Ext_A^{i-1}(S,N)
\rightarrow\Ext_A^{i}(S,M)\rightarrow \Ext_A^{i}(S,P)\\
&\rightarrow\Ext_A^{i}(S,N)\rightarrow\Ext_A^{i+1}(S,M)\rightarrow\cdots.
\end{aligned}
$$
Since $\Ext_A^{i}(S,P)=0$ for all $i<\dep_AP,$ we get
$\Ext_A^{i}(S,M)\cong\Ext_A^{i-1}(S,N)$ for all $i<\dep_AP$.
The latter is equal to $0$ for all $i\leq\dep_AN.$
In other words, for every $i<\min\{\dep_A(N)+1,\dep_A(P)\}$,
we have $\Ext_A^{i}(S,M)=0$, namely,
$$\dep_A(M)\geq \min\{\dep_A(N)+1,\dep_A(P)\}.$$
This assertion also follows from Lemma \ref{xxlem5.3}(2).
If $i=\dep_A(N)+1\leq\dep_A(P)$, one has
$$0\neq\Ext_A^{i-1}(S,N)\subseteq\Ext_A^{i}(S,M),$$
which implies that $\dep_A(M)=\dep_A(N)+1,$ as desired.
\end{proof}

\begin{lemma}
\label{xxlem5.6}
Let $A$ and $B$ be algebras. Suppose that $M$ is a finitely generated
right $B$-module and $N$ is an $(A,B)$-bimodule that is finitely
generated on both sides. Then
$$\dep_{A^{\op}}(\Hom_B(M,N))\geq\min\{2,\dep_{A^{\op}}(N)\}.$$
\end{lemma}

\begin{proof}
Consider a projective resolution of the right $B$-module $M$
$$\cdots\rightarrow P_1\rightarrow P_0\rightarrow M\rightarrow0,$$
where $P_i$ is finitely generated for $i=0,1$. By applying
$\Hom_B(-,N)$ to the exact sequence above, one has short
exact sequences
\begin{equation}
\label{E5.6.1}\tag{E5.6.1}
0\rightarrow\Hom_B(M,N)\rightarrow\Hom_B(P_0,N)\rightarrow
C_1\rightarrow 0,
\end{equation}
and
\begin{equation}
\label{E5.6.2}\tag{E5.6.2}
0\rightarrow C_1\rightarrow\Hom_B(P_1,N)\rightarrow C_2
\rightarrow 0,
\end{equation}
for some left $A$-modules $C_1$ and $C_2$. Since $P_i$ is
projective over $B$, $\Hom_B(P_i,N)$ has (left) depth at
least equal to $\dep_{A^{\op}}(N)$ for $i=0,1$. Without
loss of generality, we assume that $\dep_{A^{\op}}(N)\geq 1$.
So we consider two different cases.

If $\dep_{A^{\op}}(N)=1$, then
$\dep_{A^{\op}}(\Hom_B(P_0,N))\geq 1$. By \eqref{E5.6.1}
and Lemma \ref{xxlem5.5}, we have $\dep_{A^{\op}}(\Hom_B(M,N))
\geq 1$, as desired. If $\dep_{A^{\op}}(N)\geq 2$, then
$\dep_{A^{\op}}(\Hom_B(P_i,N))\geq 2$ for $i=0,1$.
Applying Lemma \ref{xxlem5.5} to \eqref{E5.6.2} and
\eqref{E5.6.1} respectively, we have
$\dep_{A^{\op}}(C_1)\geq 1$ and
$\dep_{A^{\op}}(\Hom_{B^{\op}}(M,N))\geq 2$.
This finishes the proof.
\end{proof}

\begin{remark}
\label{xxrem5.7}
The above lemma holds true for a finitely generated left
$B$-module $M$ and a $(B,A)$-bimodule $N$ which is finitely
generated on both sides, namely,
$$\dep_{A}(\Hom_{B^{\op}}(M,N))\geq\min\{2,\dep_{A}(N)\}.$$
\end{remark}

\begin{corollary}
\label{xxcor5.8}
Let $A$ be an algebra.
\begin{enumerate}
\item[$(1)$]
If $\dep_{A^{\op}}A\geq 2$ and $M\in\refl A^{\op},$
then $\dep_{A^{\op}}M\geq 2$.
\item[$(2)$]
If $\dep_{A}A\geq 2$ and $M\in\refl A,$ then $\dep_{A}M\geq 2.$
\end{enumerate}
\end{corollary}

\begin{proof}
We just need to show part (1). By Lemma \ref{xxlem5.6},
\begin{eqnarray*}
\dep_{A^{\op}}M&=&\dep_{A^{\op}} \Hom_{A}(\Hom_{A^{\op}}(M,A),A)\\
&\geq & \min\{2,\dep_{A^{\op}}A\}=2,
\end{eqnarray*}
as desired.
\end{proof}

The following lemma is the noncommutative version of \cite[Lemma 8.5]{IR}.

\begin{lemma}
\label{xxlem5.9}
Let $t$ be a nonnegative integer and let
$$0\rightarrow X_t\xrightarrow[]{f_t} X_{t-1}
\xrightarrow[]{f_{t-1}}\cdots\rightarrow X_2
\xrightarrow[]{f_2} X_1\xrightarrow[]{f_1}
X_0\rightarrow0$$
be an exact sequence of finitely generated $A$-modules with
$X_0\in\modu_0 A$. If,  for every $i>0$, $\dep_A X_i\geq i$,
then $X_0=0$.
\end{lemma}

\begin{proof} The assertion is automatic if Hypothesis
\ref{xxhyp5.1} fails for $A$. So for the rest of the proof, we assume
that Hypothesis \ref{xxhyp5.1} holds for $A$.

Let $Y_i$ denote $\im f_i\subseteq X_{i-1}$ for $1\leq i\leq t$.
Inductively, we will show $\dep_AY_i\geq i$ for all $i$. This is
clearly true for $i=t$. Now we assume that $\dep_AY_{i+1}\geq i+1$
for some $i$ and would like to show that $\dep_A Y_{i}\geq i$.
Consider the exact sequence
$$0\longrightarrow Y_{i+1}\longrightarrow X_i\longrightarrow Y_i
\longrightarrow 0$$
with the hypothesis $\dep_AY_{i+1}\geq i+1$ and $\dep X_i\geq i$.
By Lemma \ref{xxlem5.3}(3), $\dep_AY_i\geq i$. This finishes the
inductive step and therefore $\dep_A Y_i\geq i$ for all $1\leq i\leq t$.
In particular, $\dep_AX_0=\dep_A Y_1\geq 1$. Since $X_0=Y_1\in\modu_0 A$,
the only possibility is $X_0=0$.
\end{proof}

\begin{proposition}
\label{xxpro5.10}
Let $A$ be an algebra and $d$ be a positive integer.
Suppose that
$$X:=\,0\rightarrow X^0\xrightarrow[]{f^0}X^1
\xrightarrow[]{f^1}\cdots\rightarrow X^d
\xrightarrow[]{f^d}X^{d+1}\rightarrow\cdots$$
is a complex in $\modu A$ satisfying the following:
\begin{enumerate}
\item[$(1)$]
$\dep X^i\geq d-i$ for all $i\geq 0$;
\item[$(2)$]
$\Hm^i=0$ for all $i\geq d$, where $\Hm^i$ denotes the
$i$-th cohomology of the above complex $X$;
\item[$(3)$]
$\Hm^i\in\modu_0A$ for all $i\geq 0$.
\end{enumerate}
Then the complex $X$ is exact.
\end{proposition}

\begin{proof}
The assertion is automatic if Hypothesis
\ref{xxhyp5.1} fails for $A$. So for the rest of the proof, we assume
that Hypothesis \ref{xxhyp5.1} holds for $A$.

For $i=0$, we have an exact sequence
$$0\to \Hm^0 \to X^0\to X^0/{\Hm^0}\to 0.$$
By Lemma \ref{xxlem5.3}(2).
$\dep\Hm^0\geq\min\{\dep X^0,\dep (X^0/{\Hm^0})+1\}\geq 1$.
Since $\Hm^0\in\modu_0A,$ we have $\Hm^0=0$.

Now we fix an integer $1\leq j< d$ and assume
that $\Hm^s=0$ for all $0\leq s\leq j-1$. Then there are
two exact sequences:
$$0\rightarrow X^0\rightarrow\cdots \rightarrow X^j
\rightarrow\coker f^{j-1}\rightarrow0$$
and
$$0\rightarrow\Hm^j\rightarrow\coker f^{j-1}\rightarrow X^{j+1}.$$
By using Lemma \ref{xxlem5.3}(3) repeatedly, we obtain that
$\dep(\coker f^{j-1})\geq d-j>0$. Since $\Hm^j\in\modu_0A,$
the second exact sequence forces $\Hm^j=0$. By induction, we have
$\Hm^i=0$ for all $i=0,\cdots, d-1$ as required.
\end{proof}

\section{NQRs in dimension three}
\label{xxsec6}

Part (2) of the main theorem concerns derived equivalences of
two algebras. This can be achieved by constructing a tilting
complex between them. Let $\Lambda$ be an algebra. Recall that 
$T\in K^b(\proj \Lambda)$ is a {\it tilting complex} 
\cite[Definition 6.5]{Ri} if $\Hom_{D(\Mod \Lambda)}(T,T[i])=0$ 
for any $i\neq 0$ and the category $\add(T)$ generates 
$K^b(\proj \Lambda)$ as triangulated categories. Let $\Omega$ be 
another algebra. If there exists a tilting complex 
$T\in\K^b(\proj \Lambda)$ such that 
$\Omega\cong\End_{D(\Mod\Lambda)}(T)$, then we call $\Lambda$
and $\Omega$  {\it derived equivalent}. Rickard proved that 
there are other three equivalent conditions to characterize 
derived equivalent \cite[Theorem 6.4]{Ri}, also see 
\cite[Section 14.5]{Y2}. If a $\Lambda$-module $T$ is a tilting 
complex, then it is called a {\it tilting module}. Here we 
only need to use tilting modules, so we first recall the detailed 
definition of a tilting module. 


\begin{definition}\cite{HR}
\label{xxdef6.1}
Let $\Lambda$ be a ring. Then $T\in\modu\Lambda$ is called a
{\it tilting module} if the following conditions are satisfied:
\begin{enumerate}
\item[(a)]
$\pd_{\Lambda}T<\infty;$
\item[(b)]
$\Ext_{\Lambda}^i(T,T)=0$ for all $i>0;$
\item[(c)]
there is an exact sequence
$$0\longrightarrow\Lambda\longrightarrow T_0\longrightarrow T_1
\longrightarrow\cdots\longrightarrow T_{t-1}\longrightarrow T_t
\longrightarrow 0$$
with each $T_i\in\add T.$
\end{enumerate}
\end{definition}

Let $\Lambda$ and $\Omega$ be two algebras. If there is a tilting
$\Lambda$-module $T$ such that $\Omega\cong\End_{\Lambda}(T)$,
then  $\Lambda$ and $\Omega$ are derived equivalent, namely,
there is a triangulated equivalence between $D^b(\modu \Lambda)$
and $D^b(\modu \Omega)$ \cite[Theorem 6.4]{Ri}.

\begin{theorem}
\label{xxthm6.2}
Let $B_i$  be Auslander-regular and $\partial$-$\CM$ algebras for
$i=1,2$. Suppose $\partial(B_i)=d\geq 3$.
If there exists a $(B_1,B_2)$-bimodule $U$ satisfying
the following conditions:
\begin{enumerate}
\item[$(1)$]
$U\in\refl B_2;$
\item[$(2)$]
$\pd_{B_2} U\leq1;$
\item[$(3)$]
$\Ext^1_{B_2}(U,U)\in\modu_0B_1^{\op};$
\item[$(4)$]
$B_1\cong\End_{B_2}(U);$
\item[$(5)$]
When switching $B_1$ and $B_2,$ the above conditions still hold,
\end{enumerate}
then $U$ is a tilting $B_2$-module and further, $B_1$ and $B_2$ are
derived equivalent.
\end{theorem}

\begin{proof}
It suffices to show that $U$ is a tilting $B_2$-module as given in
Definition \ref{xxdef6.1}. Below we check (a,b,c) in Definition
\ref{xxdef6.1}.

(a) By hypothesis (2), $\pd_{B_2}U\leq 1$, hence Definition \ref{xxdef6.1}(a)
holds.

(b) By hypothesis (2), we need to prove that $\Ext_{B_2}^1(U,U)=0$. If $\modu_0 B_1^{\op}$
contains only the zero module, then hypothesis (3) implies that $\Ext_{B_2}^1(U,U)=0$.
Otherwise, Hypothesis \ref{xxhyp5.1} holds for left $B_1$-modules, which we assume for
the rest of the proof.

Consider the exact sequence
$0\longrightarrow P_1\longrightarrow P_0\longrightarrow U\longrightarrow0$
of  $B_2$-modules where $P_i$ are  projective over $B_2$. Applying
$\Hom_{B_2}(-,U)$, we obtain an exact sequence of $B_1^{\op}$-modules
$$0\rightarrow\Hom_{B_2}(U,U)\rightarrow\Hom_{B_2}(P_0,U)
\rightarrow\Hom_{B_2}(P_2,U)\rightarrow\Ext_{B_2}^1(U,U)\rightarrow 0.$$
Since $U$ is a reflexive $B_1^{\op}$-module and $\dep_{B_1^{\op}}B_1=d\geq2$
(Proposition \ref{xxpro5.4}), by Lemma \ref{xxlem5.6} and Corollary
\ref{xxcor5.8}, we have
$$\dep_{B_1^{\op}}(\Hom_{B_2}(P_0,U))\geq\min\{2,\dep_{B_1^{\op}}U\}=2,$$
and
$$\dep_{B_1^{\op}}(\Hom_{B_2}(P_1,U))\geq\min\{2,\dep_{B_1^{\op}}U\}=2\geq1.$$
Moreover, $\dep_{B_1^{\op}}(\Hom_{B_2}(U,U))=\dep_{B_1^{\op}}(B_1)=d\geq3$,
then by hypothesis (3) and Lemma \ref{xxlem5.9}, $\Ext_{B_2}^1(U,U)=0.$

(c) By hypothesis (5), we have that $B_2\cong\End_{B_1^{\op}}(U,U)$
and $\pd_{B_1^{\op}}U\leq1$. By the same proof as in (b), we have
$\Ext_{B_1^{\op}}^1(U,U)=0$. Let
$0\rightarrow Q_1\rightarrow Q_0\rightarrow U\rightarrow0$ be a
projective resolution of the $B_1^{\op}$-module $U$. Applying $\Hom_{B_1^{\op}}(-,U)$
and using the fact that $\Ext_{B_1^{\op}}^1(U,U)=0$, we obtain
an exact sequence
$$0\rightarrow B_2\rightarrow\Hom_{B_1^{\op}}(Q_0,U)\rightarrow
\Hom_{B_1^{\op}}(Q_1,U)\rightarrow0$$
of $B_2$-modules, and clearly, $\Hom_{B_1^{\op}}(Q_i,U)\in\add_{B_2}U$
for $i=0,1$. Thus we proved condition (c) of a tilting module.

Thus, $U$ is a tilting $B_2$-module, and consequently,
$B_1$, $B_2$ are derived equivalent.
\end{proof}

\begin{lemma}
\label{xxlem6.3}
Let $B_1$ and $B_2$ be algebras such that $B_2$ is
{\rm{(}}Auslander{\rm{)}}
Gorenstein. Suppose that $\partial$ satisfies $\gamma_{0,0}(B_2,B_1)^r$.
Let $M$ be a right $B_2$-module with $\pd M\leq 1$ and $U$ be
a $(B_1,B_2)$-bimodule such that $\pd U_{B_2}<\infty$.
If $\Ext_{B_2}^1(M,B_2)\in\modu_0B_2^{\op},$ then
$\Ext_{B_2}^1(M,U)\in\modu_0B_1^{\op}$.
\end{lemma}

\begin{proof}
By \cite[Section 5 (b.1)]{Le}, there is an Ischebeck spectral sequence
$$\Tor_p^{B_2}(U,\Ext_{B_2}^q(M,B_2))\Rightarrow \Ext_{B_2}^{q-p}(M,U).$$
Since $\pd M\leq 1$, the $E_2$-page of this spectral sequence has only
two nonzero columns. Therefore
$$\Tor_0^{B_2}(U,\Ext_{B_2}^1(M,B_2))\cong \Ext_{B_2}^{1}(M,U).$$
Note that
$$\Tor_0^{B_2}(U,\Ext_{B_2}^1(M,B_2))=U\otimes_{B_2}\Ext^1_{B_2}(M,B_2)
\in\modu_0{B_1}^{\op}$$
by $\gamma_{0,0}(B_2,B_1)^r$ condition. Therefore,
$\Ext_{B_2}^1(M,U)\in\modu_0 B_1^{\op}$.
\end{proof}

\begin{remark}
\label{xxrem6.4}
If $\partial=\GK$ and $B$ is affine over $\Bbbk$, then $M\in \modu_0 B$
is equivalent to $M$ being finite dimensional over $\Bbbk$. In this
case, $\partial$ automatically satisfies $\gamma_{0,i}$ for all $i$.
\end{remark}

\begin{hypothesis}
\label{xxhyp6.5}
We assume
\begin{enumerate}
\item[(1)]
Hypothesis \ref{xxhyp3.8} holds.
\item[(2)]
$\gamma_{0,0}(A,B)$ for all $A,B\in {\mathcal A}$.
\end{enumerate}
\end{hypothesis}

Next we prove a version of Theorem \ref{xxthm0.6}(2).

\begin{theorem}
\label{xxthm6.6}
Assume Hypothesis \ref{xxhyp6.5}. Let $A\in {\mathcal A}$
be an algebra with $\partial(A)=3$. Suppose
that $(B_i, _{B_i}\!(M_i)_{A}, _{A}\!(N_i)_{B_i})$ are
two $NQRs$ of $A$ for $i=1,2$. Then $B_1$ and $B_2$ are
derived equivalent.
\end{theorem}

\begin{proof} We need to verify the hypotheses in Theorem
\ref{xxthm6.2}.

By Proposition \ref{xxpro3.9}, Theorem \ref{xxthm3.15} and
Corollary \ref{xxcor2.13}, there exists a bimodule $_{B_1}\!U_{B_2}$
which is reflexive on both sides such that $B_1\cong\End_{B_2}(U)$
and $\pd_{B_2}U\leq 1$. Hence hypotheses (1,2,4) in Theorem \ref{xxthm6.2}
hold. To show hypothesis (3) in Theorem \ref{xxthm6.2}, we follow
Proposition \ref{xxpro2.18}. There are two cases that should be considered:

Case 1: $E_2^{31}=\Ext^3_{B_2^{\op}}(\Ext_{B_2}^1(U,B_2),B_2)=0$.
By Proposition \ref{xxpro2.18},
$$\Ext^i_{B_2^{\op}}(\Ext^j_{B_2}(U,B_2),B_2)=0$$
for all $(i,j)$ except for $(i,j)=(0,0)$. This implies that
$U\in\proj B_2$. Then Theorem \ref{xxthm6.2}(3) holds trivially.

Case 2: $E_2^{31}\neq 0$. Then $\Ext_{B_2}^1(U,B_2)\neq 0$, and
by Proposition \ref{xxpro2.18},
$$j_{B_2^{\op}}(\Ext_{B_2}^1(U,B_2))=3.$$
Since $B_2$ is $\partial$-$\CM$, we have
$$\partial_{B_2^{\op}}(\Ext_{B_2}^1(U,B_2))=
\partial_{B_2^{\op}}(B_2)-j_{B_2^{\op}}(\Ext_{B_2}^1(U,B_2))
=3-3=0,$$
namely, $\Ext_{B_2}^1(U,B_2)\in\modu_0B_2^{\op}$.
By Lemma \ref{xxlem6.3}, $\Ext_{B_2}^1(U,U)\in\modu_0B_1^{\op}$,
which is Theorem \ref{xxthm6.2}(3).

Up to this point, we have proved conditions (1,2,3,4) in Theorem
\ref{xxthm6.2}. By symmetry, Theorem \ref{xxthm6.2}(5) holds.
Therefore, by Theorem \ref{xxthm6.2}, $B_1$ and $B_2$ are derived
equivalent.
\end{proof}

\section{Connections between NQRs and NCCRs}
\label{xxsec7}

In this section we show that Van den Bergh's noncommutative crepant
resolutions (NCCRs) are in fact equivalent to noncommutative
quasi-resolutions (NQRs) in the commutative or central-finite case.
We use the definition given in \cite[Section 8]{IR} which is slightly
more general than original definition, see Definition \ref{xxdef0.2}.

Let $R$ be a noetherian commutative domain with finite Krull dimension.
Let ${\mathcal A}_{R,\K}$ be the category of algebras that are
module-finite $R$-algebras with $\partial$ being the Krull dimension ($\K$).
As explained in Example \ref{xxex3.1}(3), we need to specify modules too.
As usual, one-sided modules are just usual modules, but bimodules are
assumed to be $R$-central.

\begin{lemma}
\label{xxlem7.1}
Retain the notation as above. Let $A,B\in\mathcal{A}_{R,\K}$.
Then Hypothesis \ref{xxhyp1.3} holds.
\end{lemma}

\begin{proof} By \cite[Lemma 1.3]{BHZ2}, $\partial:=\K$ is exact and
symmetric. Hypothesis \ref{xxhyp1.3}(1) and (2) are clear. It remains
to show (3). By definition all bimodules are central over $R$. If
$_{A}M_B$ is finitely generated over $B$, then it is finitely
generated over $R$ as every algebra is module-finite over $R$. Then
$M$ is finitely generated over $A$. This implies that Hypothesis
\ref{xxhyp1.3}(3) is equivalent to the fact that $\partial$ is symmetric.
\end{proof}

We recall a definition from \cite[Definition 1.1(5)]{BHZ2}.
Let $A$ and $B$ be two algebras. We say $\partial$ is
{\it $(A, B)_i$-torsitive} if, for every $(A,B)$-bimodule $M$
finitely generated on both sides and every finitely generated
right $A$-module $N$, one has
\begin{equation}
\notag
\partial(\Tor^A_j(N,M)_B)\leq \partial(N_A)
\end{equation}
for all $j\leq i$. Part (1) of the following lemma was
proven in \cite{BHZ2}.

\begin{lemma}
\label{xxlem7.2}
Let $A$ and $B$ be two algebras in ${\mathcal A}_{R,\K}$.
\begin{enumerate}
\item[$(1)$] \cite[Lemma 3.1]{BHZ2}
$\partial$ is $(A,B)_{\infty}$-torsitive.
\item[$(2)$]
$\gamma_{k_1,k_2}(A,B)$ hold for all $k_1,k_2$
{\rm{(}}see Definition \ref{xxdef1.8}{\rm{)}}.
\item[$(3)$]
Hypothesis \ref{xxhyp3.8} holds.
\item[$(4)$]
Hypothesis \ref{xxhyp6.5} holds.
\end{enumerate}
\end{lemma}

\begin{proof}
(2) This follows from part (1) and the definition.

(3) This follows from Lemma \ref{xxlem7.1} and a special case of
part (2).

(4) This follows from part (3) and another special case of
part (2).
\end{proof}

For the purpose of this paper, we only need $\gamma_{0,0}(A,B)$
and $\gamma_{1,1}(A,B)$. But it is good to know that
$\gamma_{k_1,k_2}(A,B)$ hold for all $k_1,k_2$.
For the rest of this section, $\CM$ stands for ``Cohen-Macaulay''
in the classical sense in commutative algebra, while $\K$-$\CM$
is defined in Definition \ref{xxdef2.3} by taking the dimension
function $\partial$ to be the Krull dimension $\K$. By
\cite[p.1435]{BM}, when $R$ is commutative and noetherian, then
$R$ is $\K$-$\CM$ if and only if $R$ is $\CM$ and equi-codimensional.
The following lemma is known.

\begin{lemma}
\label{xxlem7.3}
Let $R$ be a commutative $d$-dimensional $\CM$ equi-codimensional
normal domain. Let $A$ be a module-finite $R$-algebra
and $K\in\modu A$. Let $s$ be an integer between $0$ and
$d-2$. Then $K\in\modu_{d-2-s} A$ if and only if
$K_{\mathfrak{p}}=0$ for every prime ideal $\mathfrak{p}$
of $R$ with $\htp(\mathfrak{p})\leq 1+s$.
\end{lemma}

\begin{proof} Since $\K M_R=\K M_A$, it suffices to
consider the case $A=R$. By \cite[Lemma 6.2.11]{MR},
we can always assume that $K$ is a critical $R$-module
such that $\mathfrak{q}:=\ann_R(K)=\{x\in R|xK=0\}$ is
a prime ideal of $R$. In this case, $K$ is an essential
$R/{\mathfrak{q}}$-module.

Suppose that $K_{\mathfrak{p}}=0$ for all prime ideals
$\mathfrak{p}$ with $\htp(\mathfrak{p})\leq 1+s$. If
$\K(K)\geq d-1-s$, by \cite[Theorem 3.1(iv)]{BM}, we have
$$\htp(\mathfrak{q})=\K R-\K R/\mathfrak{q}=\K R-\K K\leq 1+s.$$
By the definition of $\mathfrak{q}$, $K_{\mathfrak{q}}\neq 0$,
which is a contradiction. Therefore $\K(K)\leq d-2-s$, as desired.

Conversely, suppose that $\K(K)\leq d-2-s$. Then $\K R/\mathfrak{q}
\leq d-2-s$. By \cite[Theorem 3.1(iv)]{BM}, we have
$\htp(\mathfrak{q})\geq 2+s$. Therefore $K_{\mathfrak{p}}=0$ for
all prime ideal $\mathfrak{p}$ with $\htp(\mathfrak{p})\leq 1+s$.
\end{proof}

\begin{remark}
\label{xxrem7.4}
In the papers \cite{VdB2, IR}, the commutative base ring $R$ is a
normal Gorenstein domain, which is automatically CM equi-codimensional
and normal. In \cite[Theorem 1.5]{IW1}, it is assumed that $R$ is a CM
equi-codimensional normal domain. Hence the first hypothesis
of Lemma \ref{xxlem7.3} holds.
\end{remark}

\begin{proposition}
\label{xxpro7.5}
Let $R$ be a commutative noetherian $\CM$ equi-codimensional normal
domain of dimension $d$. Let $A$ be a module-finite $R$-algebra that
is a maximal $\CM$ $R$-module. If $M$ gives rise to a NCCR of $A$
in the sense of Definition \ref{xxdef0.2}(2), then
$(\Omega, _{\Omega}\!M_{A},_{A}\!(M^\vee)_{\Omega})$
is a $NQR$ of $A$. In other words,
\begin{enumerate}
\item[$(1)$]
$\Omega$ is an Auslander regular $\K$-$\CM$ algebra with
$\gldim\Omega=\K\Omega=d$.
\item[$(2)$]
$M\otimes_{A}M^\vee\cong_{d-2}\Omega$ and
$M^\vee\otimes_{\Omega}M\cong_{d-2}A.$
\end{enumerate}
\end{proposition}

\begin{proof}
(1) By the assumption, $R$ is equi-codimensional, $\Omega$ is a
module-finite $R$-algebra, and $\Omega$ is a maximal $\CM$
$R$-module, so, by \cite[Lemma 2.8(2) and Theorem 4.8]{BM},
$\Omega$ is a $\K$-$\CM$ algebra with $\K(\Omega)=\K(R)=d$.
Moreover, $\Omega$ being a nonsingular $R$-order means that it
is a homologically homogeneous noetherian PI ring. Then, by
\cite[Theorem 1.4(1)]{SZ2}, $\Omega$ is an Auslander regular
algebra with $\gldim \Omega=d$. The assertion follows.

(2) Let $\varphi: M^\vee \otimes_{\Omega}M\rightarrow A$ be the natural
evaluation map. Then there is an exact sequence
$$0\rightarrow K\rightarrow M^\vee \otimes_{\Omega}M
\xrightarrow{\varphi}A\rightarrow C\rightarrow0$$
with $K,C\in\modu  A$. By definition of a NCCR, $M\in\refl A$
is a height one progenerator of $A$, we have
$M_{\mathfrak p}^\vee \otimes_{\Omega_{\mathfrak p}}
M_{\mathfrak p}\cong A_{\mathfrak p}$ for
${\mathfrak p}\in\Spec(R)$ with $\htp({\mathfrak p})\leq 1$.
Therefore, $K_{\mathfrak p}=0=C_{\mathfrak p}$.
By Lemma \ref{xxlem7.3}, $K,C\in\modu_{d-2}R$. Combining with
the assumption that $A$ is a module-finite $R$-algebra,
$K,C\in\modu_{d-2}A$, namely,
$M^\vee \otimes_{\Omega}M\cong_{d-2}A$.

Similarly, there is a natural map
$$\alpha: M\otimes_{A}M^\vee\to \Hom_A(M,M)=:\Omega$$
such that $\alpha(n\otimes f)(m)=nf(m)$ for all
$f\in M^{\vee}$ and all $n,m\in M$. One can use the
above argument to show that
$M\otimes_{A}M^\vee\cong_{d-2}\Omega$, whence (2) follows.
\end{proof}

Conversely, a NQR is also a NCCR for Gorenstein singularities.

\begin{proposition}
\label{xxpro7.6}
Let $R$ be a commutative $d$-dimensional $\CM$ equi-codimensional
normal domain. Let $A$ be a module-finite $R$-algebra
that is a maximal $\CM$ $R$-module.  Suppose that $A$ is
Auslander-Gorenstein and $\K$-$\CM$, then a NQR of $A$, say $(B, M,N)$,
provides a NCCR $B$ of $A$ in the sense of Definition \ref{xxdef0.2}.
\end{proposition}

\begin{proof} Let $(B,M,N)$ be a NQR of $A$. Then $B$ is an Auslander regular
$\K$-CM algebra of Krull dimension $d$, and
$$M\otimes_A N\cong_{d-2} B, \quad
N\otimes_B M\cong_{d-2} A.$$
By Proposition \ref{xxpro3.9}, $(M,N)$ can be replaced by $(U,V)$
such that $(B,U,V)$ is also a NQR of $A$ and that $U$ and $V$ are reflexive
on both sides. By Lemmas \ref{xxlem3.13} and \ref{xxlem7.2},
$B\cong \End_{A}(U)$. Since $B$ is Auslander regular and
$\K$-CM, it is easy to check that $B$ is a non-singular order. It remains
to show that $U$ is a height one progenerator. By Proposition \ref{xxpro3.9},
we have
$$U\otimes_A V\cong_{d-2} B \quad {\text{and}} \quad
V\otimes_B U\cong_{d-2} A.$$
It follows from Lemma \ref{xxlem7.3} that
$$U_{\mathfrak{p}}\otimes_{A_{\mathfrak{p}}} V_{\mathfrak{p}}
\cong B_{\mathfrak{p}} \quad {\text{and}} \quad
V_{\mathfrak{p}}\otimes_{B_ {\mathfrak{p}}}U_{\mathfrak{p}}
\cong A_{\mathfrak{p}}$$
for every prime ideal $\mathfrak{p}$ of $R$ with
$\htp(\mathfrak{p})\leq 1$. Hence $U$ is a height one
progenerator of $A$. Therefore $U$ gives a NCCR of $A$.
\end{proof}

By the above two propositions, NCCRs are essentially equivalent to NQRs
when $A$ is Auslander-Gorenstein. Therefore Theorem \ref{xxthm0.4}(2b)
is essentially equivalent to Theorem \ref{xxthm6.6} in this setting.
In the next section, we will introduce more examples of NQRs in the
noncommutative setting. One advantage of NQRs is that they can be
defined for many algebras that are not Gorenstein (not even CM). In
this case we do not require $M$ or $N$ to be reflexive. Here is an
easy example.

\begin{example}
\label{xxex7.7}
Let $B$ be the commutative polynomial ring $\Bbbk[x_1,x_2,x_3]$
with the standard grading
and let $A$ be the subring of $B$ generated by $B_{\geq 2}$. Then
$A$ is noetherian of Krull dimension three and the Hilbert series of
$A$ is
$$H_A(t)=\frac{1}{(1-t)^3}-3t.$$
It is easy to see that $A$ is not normal and that $H_A(t)$ can not
be written as
$$\frac{1+a_1 t+a_2 t^2+\cdots +a_d t^d}{(1-t^{n_1})(1-t^{n_2})(1-t^{n_3})}$$
for any nonnegative integers $a_i, n_j$. By
\cite[Ex.21.17(b), p.551]{Ei}, $A$ is not CM. Then a NCCR of $A$
is not defined. But $A$ has a NQR as shown below.

Let $M=B$ as a $(B,A)$-bimodule and $N=B$ as an $(A,B)$-bimodule.
It is easy to see that
$$M\otimes_A N\cong_0 B, \quad
{\text{and}}\quad
N\otimes_B M\cong_0 A.$$
Consequently, $(B,M,N)$ is a NQR of $A$. As a consequence
of Proposition \ref{xxpro7.9} below, any two NQRs of $A$ are
Morita equivalent.

In addition, it is easy to check that $A$ has an isolated singularity
at the unique maximal graded ideal.
\end{example}

Though NQRs are weaker than NCCRs, not every algebra admits a NQR.

\begin{example}
\label{xxex7.8}
Let $A$ be an affine commutative Gorenstein algebra that does not
admit a NCCR. Then $A$ does not admit a NQR by Proposition
\ref{xxpro7.6}. For example, $A$ is an affine Gorenstein algebra
of dimension two with non-isolated singularities, then $A$ does
not admit either a NCCR or a NQR. By \cite[Theorem 1.2(2) and
Example 3.5]{Da}, there are isolated hypersurface singularities
of (any) even dimension $\geq 4$ that do not admit either a NCCR
or a NQR.
\end{example}

\begin{proposition}
\label{xxpro7.9}
Let $A$ be a module-finite  $R$-algebra with $\partial(A)=3$.
Suppose
\begin{enumerate}
\item[$(a)$]
$(B_i, _{B_i}\!(M_i)_{A}, _{A}\!(N_i)_{B_i})$ are two $NQRs$ of
$A$ for $i=1,2$, and
\item[$(b)$]
$B_1$ or $B_2$ is Azumaya.
\end{enumerate}
Then $B_1$ and $B_2$ are Morita equivalent.
\end{proposition}

\begin{proof} The hypotheses in Theorem \ref{xxthm6.6}
are automatic by Lemma \ref{xxlem7.2}. Hence $B_1$ and $B_2$
are derived equivalent. Since $B_1$ (or $B_2$) is Azumaya, by
\cite[Proposition 5.1]{YZ1}, $B_1$ and $B_2$ are Morita
equivalent.
\end{proof}

\section{Examples of NQRs of noncommutative algebras and comments}
\label{xxsec8}

In this section we give some examples of NQRs of noncommutative algebras.
At the end of the section we also give some comments.
It turns out that, except for Example \ref{xxex8.6}, all examples in this
section have the same kind of construction, namely, by noncommutative
McKay correspondence. Precisely, fixed subrings $R^H$, considered as
noncommutative quotient singularities, have NQRs of the form $R\# H$,
where $R$ and $H$ will be explained in details. However, by taking
different $R$ and $H$, we obtain many different examples.

\subsection{Graded case}
\label{xxsec8.1}
Let ${\mathcal A}_{gr,\GK}$ be the category of locally finite
${\mathbb N}$-graded noetherian algebras with finite Gelfand-Kirillov
dimension and let $\partial=\GK$. Modules are usual ${\mathbb Z}$-graded
$A$-modules. In this setting, $(A,B)$-bimodules are assumed to be
${\mathbb Z}$-graded $(A,B)$-bimodules, namely, having both a left
graded $A$-module and a right graded $B$-module structure with the
same grading. See Example \ref{xxex3.1}(1).

\begin{remark}
\label{xxrem8.1}
Many of basic results in ring theory and module theory have been
generalized to the graded setting in the literature. For example,
the graded version of some basic results in ring theory can be
found in the book \cite{NvO}. Using the graded version of these
results one can carefully adapt the arguments to reprove all
statements in Sections 1-7 in the graded setting. To save space
we will use the graded version of results in Section 1-7 without
proofs.
\end{remark}

\begin{lemma}
\label{xxlem8.2}
Retain the notation as above concerning the category
${\mathcal A}_{gr,\GK}$.
\begin{enumerate}
\item[$(1)$]
Hypothesis \ref{xxhyp1.3} holds.
\item[$(2)$]
Let $A$ and $B$ be two algebras in ${\mathcal A}_{gr,\GK}$.
Then $\partial$ is $(A,B)_{\infty}$-torsitive.
As a consequence, $\gamma_{k_1,k_2}(A,B)$ hold for all
$k_1, k_2$.
\item[$(3)$]
Hypothesis \ref{xxhyp6.5} holds.
\end{enumerate}
\end{lemma}

\begin{proof} (1) By \cite[Lemma 1.2(1)]{BHZ2}, $\partial$ is exact.
By \cite[Lemma 1.2(4)]{BHZ2}, it is symmetric.  Hypothesis 1.3(2)
is \cite[Lemma 5.3(b)]{KL}.

(2) This is \cite[Lemma 1.2(6)]{BHZ2}.

(3) This follows from parts (1) and (2) and the definition.
\end{proof}

From now on until Example \ref{xxex8.7}, we are working
with the category ${\mathcal A}_{gr,\GK}$.

\begin{proof}[Proof of Theorem \ref{xxthm0.6}]
By Lemma \ref{xxlem8.2}(3), Hypothesis \ref{xxhyp6.5} holds.

(1) This is a (graded) consequence of Theorem \ref{xxthm4.2}.

(2) This is a consequence of Theorem \ref{xxthm6.6} and
Lemma \ref{xxlem8.2}(3).
\end{proof}

\begin{proposition}
\label{xxpro8.3}
Suppose that $A$ and $B$ are two noetherian locally finite
${\mathbb N}$-graded algebras that satisfy the following
\begin{enumerate}
\item[$(a)$]
$B$ is an Auslander regular $\CM$ algebra with $\GK(B):=d\geq 2$.
\item[$(b)$]
Let $e$ be an idempotent in $B$ {\rm{(}}or in $B_0${\rm{)}} and $A=eBe.$
\item[$(c)$]
$N:=eB$ is an $(A,B)$-bimodule which is finitely generated  on both sides.
\item[$(d)$]
$M:=Be$ is a $(B,A)$-bimodule which is finitely generated  on both sides.
\item[$(e)$]
$\GK(B/BeB)\leq d-2.$
\end{enumerate}
Then $(B,M,N)$ is a $NQR$ of $A$.
\end{proposition}

\begin{proof}
Below is a proof in the ungraded setting which can
easily adapted to the graded case. Since $\partial=\GK,$
by \cite[Lemma 2.2(ii)]{BHZ2}, $\partial((Ne)_A)\leq
\partial(N_B)$ for every finitely generated right
$B$-module $N$, which is precisely
\cite[Hypothesis 2.1(7)]{BHZ2}. It is easy to verify
\cite[Hypothesis 2.1(1-6)]{BHZ1}. By Lemma \ref{xxlem8.2}(2),
$\partial$ satisfies $\gamma_{d-2,1}(eB),$ which is precisely
\cite[(E2.3.1)]{BHZ1}. Therefore we can apply the proof of
\cite[Lemma 2.3]{BHZ1}. By the proof of \cite[Lemma 2.3]{BHZ1},
the hypothesis $\GK(B/BeB)\leq d-2$ implies that
$$M\otimes_A N\cong_{d-2} B.$$
On the other hand, it is clear that
$$N\otimes_{B} M=eBe=A.$$
Therefore $(B,M,N)$ is a NQR of $A$.
\end{proof}

\begin{remark}
\label{xxrem8.4}
\begin{enumerate}
\item[(1)]
By the above proof, Proposition \ref{xxpro8.3} holds in the
ungraded case as long as condition \cite[(E2.3.1)]{BHZ1} holds.
Note that \cite[(E2.3.1)]{BHZ1} is a consequence of
$\partial$ being $(A,B)_{\infty}$-torsitive for those
algebras $A$ and $B$ in the category ${\mathcal A}$.
\item[(2)]
Similar to the proof of Proposition \ref{xxpro8.3},
if we replace condition (e) by
$$\GK (B/BeB)\leq d-2-s,$$
then $(B,M,N)$ is an $s$-NQR of $A$ in the sense of
Definition \ref{xxdef3.2}(1).
\item[(3)]
If $A=:R$ is a commutative normal Gorenstein domain, then
the NQRs in Proposition \ref{xxpro8.3} might not be in
the category of ${\mathcal A}_{R, \K}$ (Section \ref{xxsec7})
as we are not require that $B$ is $R$-central. Keep this
in mind, it is also possible that $R$ has a NCCR in the category
${\mathcal A}_{R, \K}$ and a NQR not in ${\mathcal A}_{R, \K}$
(but in a different category such as ${\mathcal A}_{gr,\GK}$).
\end{enumerate}
\end{remark}

Explicit examples of NQRs in the graded case are given next.

\begin{example}
\label{xxex8.5}
Suppose the following hold.
\begin{enumerate}
\item[(a)]
Let $R$ be a noetherian connected graded (locally finite) Auslander regular
$\CM$ algebra with $\GK(R)=d\geq 2$.
\item[(b)]
Let $H$ be a semisimple Hopf algebra acting on $R$ homogeneously
and inner-faithfully with integral $\int$ such that
$\varepsilon(\int)=1$.
\item[(c)]
Let $B$ be the smash product algebra $R\# H$ with $e:=1\#\int\in B$
and $A$ be the fixed subring $R^H$.
\item[(d)]
Suppose $\GK(B/BeB)\leq d-2.$
\end{enumerate}
By Proposition \ref{xxpro8.3},
$(B,M,N):=(B,Be,eB)$ is a NQR of $A$. This produces
many examples of NQRs in following work. Note that the condition
(d) is equivalent to that a version of Auslander's theorem holds,
namely, the natural algebra morphism
\begin{equation}
\notag
\phi: \quad R\# H\longrightarrow \End_{R^H}(R)
\end{equation}
is an isomorphism. In \cite{CKWZ1, CKWZ2, BHZ1, BHZ2, CKZ, GKMW},
the results state that Auslander's theorem holds instead of
condition (d). Auslander's theorem
is a fundamental ingredient in the study of the McKay correspondence,
see \cite{CKWZ1, CKWZ2}. In the following we further assume that
${\text{char}}\; \Bbbk=0$.
\begin{enumerate}
\item[(1)]
Let $R$ be an Auslander regular and $\CM$ algebra of global dimension
two and $H$ act on $R$ with trivial homological determinant.
Then $A:=R^H$ has a NQR \cite[Theorem 0.3]{CKWZ1}.
\item[(2)]
Let $R$ be a graded noetherian
down-up algebra (of global dimension three) which is not
$A(\alpha,-1)$ and $G$ be a finite subgroup of $\Aut_{gr}(R)$.
Then $A:=R^G$ has a NQR \cite[Theorem 0.6]{BHZ2}.
\item[(3)]
Let $R$ be a graded noetherian down-up algebra (of global
dimension three) and $G$ be a finite subgroup coacting on
$R$ with trivial homological determinant. Then $A:=R^{co \;G}$ has a
NQR \cite[Theorem 0.1]{CKZ}.
\item[(4)]
Let $R=\Bbbk_{-1}[x_1,\cdots,x_n]$ and ${\mathbb S}_n$ act on
$R$ naturally permuting variables $x_i$. Then $A:=R^{G}$ has a
NQR for every nontrivial subgroup $G\subseteq {\mathbb S}_n$
\cite[Theorem 2.4]{GKMW}. A special case was proved earlier
in \cite[Theorem 0.5]{BHZ1}.
\end{enumerate}
\end{example}

Next we give an example of a NQR that does not fit into the framework
of Proposition \ref{xxpro8.3}.

\begin{example}
\label{xxex8.6}
Let $q$ be a nonzero scalar in $\Bbbk$ that is not a root of unity.
Let $B$ be the algebra $\Bbbk\langle x,y\rangle/(yx-qxy-x^2)$,
which is connected graded noetherian Auslander regular and
$\CM$ of $\GK$ 2. Let $A:=\Bbbk+ By$ be the subalgebra of $B$
as given in \cite[Notation 2.1]{SZ1}. By \cite[Theorem 2.3]{SZ1},
$A$ is a noetherian algebra that does not satisfy the condition
$\chi$ in the sense of \cite{AZ}. As a consequence, $A$ does
not admit a balanced dualizing complex in the sense of Yekutieli
\cite{Y1}. In other words, this algebra does not have nice
properties required in noncommutative algebraic projective geometry.
By \cite[Corollary 2.8]{SZ1},
\begin{equation}
\notag
\qmod_0 A\cong \qmod_0 B.
\end{equation}
This indicates that $A$ might have a NQR. Indeed, this is the case
as we show next.

Let $M$ be the (graded) $(B,A)$-bimodule $By$ and $N$ be the (graded)
$(A,B)$-bimodule $B$. One can verify that $M$ and $N$ are finitely
generated on both sides. Note that, as a right $A$-module, $M\cong_0
A$ since we have an exact sequence $0\to M\to A\to \Bbbk\to 0$.
Hence, following the Hilbert series computations,
$$N\otimes_B M\cong_0 B\otimes_B By\cong By\cong_0 A$$
as $A$-bimodules, and
$$M\otimes_A N\cong_0 By\otimes_A B\cong_0 ByB\cong_0 B$$
as $B$-bimodules, where the last $\cong_0$ follows from the fact
that $ByB$ is co-finite-dimensional inside $B$. By definition,
$B$ is a NQR of $A$.
\end{example}

\subsection{Ungraded case}
\label{xxsec8.2}

All NQRs in the graded case are NQRs in the ungraded setting.
Below are other ungraded examples.

\begin{example}
\label{xxex8.7}
Let ${\mathcal A}_{ungr,\GK}$ be the category of affine $\Bbbk$-algebras
with finite Gelfand-Kirillov dimension and let $\partial = \GK$.
Suppose the following hold.
\begin{enumerate}
\item[(a)]
Let $R$ be a noetherian Auslander regular
$\CM$ algebra with $\GK(R)=d\geq 2$.
\item[(b)]
Let $H$ be a semisimple Hopf algebra acting on $R$
and inner-faithfully with integral $\int$ such that
$\varepsilon(\int)=1$.
\item[(c)]
Let $B$ be the smash product algebra $R\# H$ with $e:=1\#\int\in B$
and $A$ be the fixed subring $R^H$.
\item[(d)]
Suppose $\GK(B/BeB)\leq d-2.$
\end{enumerate}
If \cite[(E2.3.1)]{BHZ1} holds, by Proposition \ref{xxpro8.3},
$(B,M,N):=(B,(R\# H)e,e(R\# H))$ is a NQR of $A$. We have
some examples in the ungraded case. Here is the
first example. Again assume that ${\text{char}}\; \Bbbk=0$.
Let $R$ be the universal enveloping algebra $U({\mathfrak g})$
of a finite dimensional Lie algebra ${\mathfrak g}$. Suppose
that ${\mathfrak g} \neq {\mathfrak g}' \ltimes \Bbbk x$ for
a 1-dimensional Lie ideal $\Bbbk x\subseteq {\mathfrak g}$
and and a Lie subalgebra ${\mathfrak g}'\subset {\mathfrak g}$.
Then \cite[Corollary 0.5]{BHZ2} implies that $R^G$ has a NQR
for every finite group $G\subseteq \Aut_{Lie}({\mathfrak g})$.
\end{example}

\begin{example}
\label{xxex8.8} Let ${\mathcal A}_{PI, \GK}$ be the category of
affine $\Bbbk$-algebras that satisfy a polynomial identity and let
$\partial = \GK$. In fact, $\GK=\K$ in this case. But we do not assume
that algebras are central-finite. By \cite[Lemma 3.1]{BHZ2},
$\partial$ is $(A,B)_{\infty}$-torsitive for two algebras $A$ and $B$ in
${\mathcal A}_{ungr,PI}$. As a consequence, $\gamma_{k_1,k_2}(A,B)$
hold for all $k_1, k_2$, see Lemma \ref{xxlem7.2}. Then in the
setting of Example \ref{xxex8.6}, $R^H$ has a NQR $B:=R\# H$
when $\partial(B/BeB)\leq \partial(B)-2$, or equivalently,
when the Auslander's theorem holds by \cite[Theorem 3.3]{BHZ2}.
Explicit examples of $R$ and $H$ are given in
\cite[Corollaries 3.4 and 3.7]{BHZ2}.
\end{example}

\subsection{Comments on potential directions of further research}
\label{xxsec8.3}

There are some further studies of NQRs in dimension two in 
\cite[Theorem 0.2(1)]{QWZ} where the Gabriel quiver of a
NQR is classified. It is natural to ask what we can do in 
dimension three.

The next question was suggested by the referee. 
Is it possible to remove the condition that the singular 
ring $A$ (either in the commutative regime or in the 
noncommutative regime) is Gorenstein or CM? Instead, we 
just assume that $A$ has an Auslander dualizing complex. 
Note that, except for the definition (Definition \ref{xxdef3.2}), 
we use the Gorenstein or CM property in a large part 
of the paper. 

In future study of higher dimensional NQRs, dualizing complexes
and derived categories \cite{YZ2, Y2} would play a more important
role than the classical methods presented in this paper.

\subsection*{Acknowledgments}

The authors thank Xiaowei Xu and Zhibing Zhao for their careful 
reading of a draft of the paper and useful suggestions and thank 
the referees for their very careful reading and extremely valuable 
comments. X.-S. Qin was partially supported by the Foundation of 
China Scholarship Council (Grant No. [2016]3100). Y.-H. Wang was 
partially supported by the Foundation of China Scholarship Council
(Grant No. [2016]3009), the Foundation of Shanghai Science and 
Technology Committee (Grant No. 15511107300), the Natural Science 
Foundation of China (No. 11871071). Y.H. Wang and X.-S. Qin thank 
the Department of Mathematics, University of Washington for its 
very supportive hospitality during their visits. J.J. Zhang was 
partially supported by the US National Science Foundation (Grant 
Nos. DMS-1402863 and DMS-1700825).

\end{document}